\documentclass[english]{amsart}

\usepackage{amsmath}
\usepackage{amssymb}

\usepackage{amscd}
\usepackage{tabularx}

\usepackage[all,cmtip,line]{xy}

\newcommand{\F}{\mathbb{F}}

\DeclareMathOperator{\Frob}{Frob}

\newcommand{\To}{\longrightarrow}
\newcommand{\into}{\hookrightarrow}
\newcommand{\onto}{\twoheadrightarrow}
\newcommand{\isoto}{\stackrel{\sim}{\To}}
\newcommand{\p}{\mathfrak{p}}

\newcommand{\M}{\mathcal{M}}

\newcommand{\bigO}{\mathcal{O}}

\newcommand{\Z}{\mathbb{Z}}
\newcommand{\A}{\mathbb{A}}
\newcommand{\Q}{\mathbb{Q}}

\newcommand{\T}{\mathbb{T}}

\newcommand{\rhobar}{\overline{\rho}}

\newcommand{\Gal}{\operatorname{Gal}}
\newcommand{\GL}{\operatorname{GL}}
\newcommand{\PGL}{\operatorname{PGL}}
\newcommand{\HT}{\mathrm{HT}}

\newcommand{\Qbar}{\overline{\Q}}
\newcommand{\Qp}{\Q_p}
\newcommand{\Fp}{\F_p}

\newcommand{\Qpbar}{\overline{\Q}_p}
\newcommand{\Fpbar}{\overline{\F}_p}

\newcommand{\BrMod}{\operatorname{BrMod}}

\newcommand{\Ind}{\operatorname{Ind}}

\newcommand{\End}{\operatorname{End}}

\newcommand{\Aut}{\operatorname{Aut}}

\newcommand{\Fil}{\mathrm{Fil}}

\newcommand{\Sym}{\operatorname{Sym}}

\newcommand{\om}{\tilde{\omega}}
\newcommand{\omt}{\om_2}
\newcommand{\OO}{\calO}
\newcommand{\m}{\frakm}
\newcommand{\calO}{{\mathcal O}}
\newcommand{\frakm}{\mathfrak{m}}
\newcommand{\Zp}{\Z_{p}}
\newcommand{\st}{\mathrm{st}}
\DeclareMathOperator{\im}{im}
\newcommand{\G}{{\mathcal G}}
\newcommand{\univ}{\operatorname{univ}}
\newcommand{\dd}{\operatorname{dd}}

\usepackage{amsthm}

\newtheorem{thm}{Theorem}[section]
\newtheorem{corollary}[thm]{Corollary}
\newtheorem{cor}[thm]{Corollary}

\newtheorem{lem}[thm]{Lemma}
\newtheorem{prop}[thm]{Proposition}
 \theoremstyle{definition}
\newtheorem{defn}[thm]{Definition} \theoremstyle{remark}
\newtheorem{rem}[thm]{Remark} \numberwithin{equation}{subsection}

\begin{document}
\title[Serre weights: the totally ramified case]{Serre weights for mod $p$ Hilbert modular forms: the totally ramified case}
\author{Toby Gee} \email{tgee@math.harvard.edu} \address{Department of
  Mathematics, Harvard University}
\author{David Savitt} \email{savitt@math.arizona.edu}
\address{Department of Mathematics, University of Arizona}
\thanks{The authors were partially
  supported by NSF grants DMS-0841491 and DMS-0600871}

\begin{abstract}We study the possible weights of an irreducible 2-dimensional
modular mod $p$  representation of $\Gal(\overline{F}/F)$, where $F$
is a totally real field which is totally ramified at $p$, and the
representation is tamely ramified at the prime above $p$.  In most
cases we determine the precise list of possible weights; in the
remaining cases we determine the possible weights up to a short and
explicit list of exceptions.
\end{abstract}

\maketitle \tableofcontents
\section{Introduction}Let $p$ be a prime number. The study of the possible weights of a mod $p$
modular Galois representation was initiated by Serre in his famous
paper \cite{ser87}. This proposed a concrete conjecture (``the
weight part of Serre's conjecture'') relating the weights to the
restriction of the Galois representation to an inertia subgroup at
$p$. This conjecture was resolved (at least for $p>2$) by work of
Coleman-Voloch, Edixhoven and Gross (see \cite{MR1176206}).

More recently the analogous questions for Hilbert modular forms have
been a focus of much investigation, beginning with the seminal paper
\cite{bdj}. Let $F$ be a totally real field with absolute Galois group
$G_F$. Then to any irreducible modular
representation $$\rhobar:G_F\to\GL_2(\Fpbar)$$ there is associated a
set of weights $W(\rhobar)$, the set of weights in which $\rhobar$ is
modular (see section \ref{sec: notation} for the definitions of
weights and of what it means for $\rhobar$ to be modular of a certain
weight). Under the assumption that $p$ is unramified in $F$ the paper
\cite{bdj} associated to $\rhobar$ a set of weights $W^{?}(\rhobar)$, and
conjectured that $W^?(\rhobar)=W(\rhobar)$. Many cases of this
conjecture were proved in \cite{gee053}.

The set $W^?(\rhobar)$ depends only on the restrictions of $\rhobar$
to inertia subgroups at places dividing $p$. In the case that these
restrictions are tamely ramified, the conjecture is completely
explicit, while in the general case the set depends on some rather
delicate questions involving extensions of crystalline characters.

Schein \cite{scheinramified} has proposed a generalisation of the
conjecture of \cite{bdj} in the tame case, removing the restriction
that $p$ be unramified in $F$.  Not much is currently known about
this conjecture; \cite{scheinramified} proves some results towards
the implication $W(\rhobar)\subset W^?(\rhobar)$, but very little is
known about the harder converse implication in the case that $p$ is
ramified. It is clear that the techniques of \cite{gee053} will not
on their own extend to the general case, as they rely on
combinatorial results which are false if $p$ ramifies.

In this paper, we prove most cases of the conjecture of
\cite{scheinramified} in the case that~$p$ is totally ramified in
$F$. Our techniques do not depend on the fact that there is only a
single prime of $F$ above $p$, and they would extend to the case where
every prime of $F$ above $p$ has residue field $\Fp$ (or in
combination with the techniques of \cite{gee053}, to the case where
every prime of $F$ above $p$ is either unramified or totally
ramified). We have restricted to the case that $p$ is totally ramified
in order to simplify the exposition.

We assume throughout that $p$ is odd, and that
$\rhobar|_{G_{F(\zeta_p)}}$ is irreducible. We make a mild
additional assumption if $p=5$.  All of these restrictions are
imposed by our use of the modularity lifting theorems of
\cite{kis04} (or rather, by their use in \cite{gee061}). Let
$(p)=\p^e$ in $\bigO_F$, where $e=[F:\Q]$. Under these assumptions,
we are able to prove that if $\rhobar|_{G_{F_\p}}$ is irreducible,
then $W(\rhobar)=W^?(\rhobar)$. If $\rhobar|_{G_{F_\p}}$ is a sum of
two characters, then we show that $W(\rhobar)\subset W^?(\rhobar)$, and
that equality holds if $e\geq p$. If $e\leq p-1$ then we
prove
that the weights in $W^{?}(\rhobar)$ all occur except that we miss
between zero and four weights; under the extra hypothesis that $\rhobar$ has
an ordinary modular lift, we can usually (but not quite always) treat
these exceptions as well.

We establish that $W(\rhobar)\subset W^?(\rhobar)$ by a computation
using Breuil modules with descent data, in the same style as analogous
computations in the literature; we have to use a few tricks in
boundary cases, but these arguments are more or less standard.

For the harder converse, our techniques are roughly a
combination of those of \cite{gee053} and an argument due to Kevin
Buzzard, which uses a technique known as ``weight cycling''. This
argument was first written up in section 5 of \cite{taymero} in the
case that $p$ splits completely in $F$. The argument essentially
depends only on the residue field of primes dividing $p$, and thus
applies equally well in our totally ramified setting. It is the use
of this argument that entails our restriction to the totally
ramified case, rather than permitting arbitrary ramification. We note
that while it might be possible to prove some additional results for
small residue fields in cases with little ramification, it is clear
that if one wishes to work with arbitrarily ramified fields, then
our methods will only give complete results in the case of residue
field $\F_p$ (for example, there will be many situations where even if one
is told that $\rhobar$ is modular of all but one weight in
$W^?(\rhobar)$, our methods will be unable to prove modularity in this
remaining weight).

As in \cite{gee053}, the plan is to construct modular lifts of
$\rhobar$ which are potentially Barsotti-Tate of specific type, using
the techniques of Khare-Wintenberger, as explained in
\cite{gee061}. These techniques reduce the construction of such lifts
to the purely local problem of exhibiting a single potentially
Barsotti-Tate lift of
$\rhobar|_{G_{F_\p}}$ of the appropriate type. In the case that
$\rhobar|_{G_{F_\p}}$ is irreducible, writing down such a lift is
rather non-trivial; in fact, as far as we are aware, no-one has
written down such a lift in any case in which $e>1$. We accomplish
this by means of an explicit construction of a corresponding strongly
divisible module.

The immediate consequence of the existence of these lifts is that
$\rhobar$ is modular of one of two weights, the constituents of a
certain principal series representation. In \cite{gee053} we were able
to conclude that only one of these two weights was actually possible,
but in the totally ramified case $\rhobar$ is frequently modular of
both of these weights, so no such argument is possible. It is at this
point that we employ weight cycling. Crucially, we can frequently
ensure that our lift is non-ordinary, and when this holds weight
cycling ensures that $\rhobar$ is modular of both weights. The cases
where we cannot guarantee a non-ordinary lift are certain of those for which
$\rhobar|_{G_{F_\p}}$ is reducible and $e<p$, which is why our results
are slightly weaker in this case.

We note that our methods should also be applicable in the non-tame
case, and should give similar results, subject to the appropriate
local calculations. For explicit conjectures in this case
(``explicit'' in the sense of \cite{bdj}, i.e., in terms of certain
crystalline extensions) see the forthcoming \cite{geesavitt}.

We now detail the outline of the paper. In section \ref{sec:
notation} we give our initial definitions and notation. In
particular, we introduce spaces of algebraic modular forms on
definite quaternion algebras, and we explain what it means for
$\rhobar$ to be modular of a specific weight. Note that we work
throughout with these spaces of forms, rather than their analogues
for indefinite quaternion algebras as used in \cite{scheinramified}
or \cite{bdj}. While our results do not immediately go over to their
setting, our proofs do; both the results on the existence of
Barsotti-Tate lifts of specified type and the weight cycling
argument are available in that case (for the latter, see
\cite{MR2407230}). We restrict ourselves to the case of definite
quaternion algebras for ease of exposition, for example in order to
use the results of \cite{gee061} as a black box.

In section \ref{sec:tame lifts} we explain which tame lifts we will
need to consider, and the relationship between the existence of
modular lifts of specified types and the property of being modular
of a certain weight. This amounts to recalling certain concrete
instances of the local Langlands correspondence for $\GL_2$ and
local-global compatibility. All of this material is
completely standard.

We give an exposition of the weight cycling result in section
\ref{sec: weight cycling}, adapted to the situation at hand. In
particular, we combine weight cycling with the results of earlier
sections to give a result establishing that $\rhobar$ is
modular of a particular weight provided that it has a modular lift
which is potentially Barsotti-Tate  of a particular type and is
non-ordinary.

Having done this, we now need some concrete results on the existence
of (local) potential Barsotti-Tate representations of particular
type that lift $\rhobar |_{G_{F_\p}}$. We warm up for these calculations by establishing (in the tame
case) the inclusion $W(\rhobar)\subset W^?(\rhobar)$ in section
\ref{sec: necessity}. This uses a calculation with Breuil modules. In
section \ref{sec:niveau-1-case}
we produce the required lifts in
the case that $\rhobar|_{G_{F_\p}}$ is reducible. This case is
relatively straightforward, as we are able to use reducible lifts. The
irreducible case is considerably more challenging, and is completed in
sections \ref{sec:the lifting result in the
  residually irred case} and \ref{sec:niveau_two_conclusion}, where we explicitly construct the lifts by
writing down the corresponding strongly divisible modules.

Finally, in section \ref{main results} we combine these results with
the lifting techniques of \cite{gee061} and some combinatorial
arguments to prove the main theorems.

We are grateful to Florian Herzig and to the anonymous referee for a careful reading of the paper.

\section{Notation and assumptions}\label{sec: notation}Let $p$ be an odd prime. Fix an
algebraic closure $\Qbar$ of $\Q$, an algebraic closure  $\Qpbar$ of
$\Qp$, and an embedding $\Qbar\into\Qpbar$. We will consider all
finite extensions of $\Q$ (respectively $\Qp$) to be contained in
$\Qbar$ (respectively $\Qpbar$). If $K$ is such an extension, we let
$G_{K}$ denote its absolute Galois group $\Gal(\overline{K}/K)$. Let $F$ be a totally real field in which $p$ is
totally ramified, say $(p)=\p^{e}$. Choose a uniformiser
$\pi_{\p}\in\p$. Let $\rhobar:G_{F}\to\GL_{2}(\Fpbar)$ be a continuous
Galois representation. Assume from now on that
$\rhobar|_{G_{F(\zeta_{p})}}$ is absolutely irreducible. If $p=5$ and
the projective image of $\rhobar$ is isomorphic to $\PGL_{2}(\F_{5})$,
assume further that $[F(\zeta_{5}):F]=4$. We normalise the
isomorphisms of local class field theory so that a uniformiser
corresponds to a geometric Frobenius element.

We wish to discuss the Serre weights of $\rhobar$. We choose to work
with totally definite quaternion algebras. We recall the basic
definitions and results that we need, adapted to the particular case
where $F$ is totally ramified at $p$.

Let $D$ be a quaternion algebra with center $F$ which is ramified at all infinite places of $F$ and at a set $\Sigma$ of finite places, which does not contain $\p$. Fix a maximal order $\bigO_D$ of $D$ and for each finite place $v\notin\Sigma$ fix an isomorphism $(\bigO_D)_v\isoto M_2(\bigO_{F_v})$. For any finite place $v$ let $\pi_{v}$ denote a uniformiser of $F_v$.

Let $U=\prod_v U_v\subset (D\otimes_F\mathbb{A}^f_F)^\times$ be a
compact open subgroup, with each $U_v\subset
(\bigO_D)^\times_v$. Furthermore, assume that $U_v=(\bigO_D)_v^\times$
for all $v\in\Sigma$.

Take $A$ a topological $\Z_p$-algebra. Fix a continuous representation
$\sigma:U_{\p}\to\Aut(W_{\sigma})$ with $W_{\sigma}$ a finite free
$A$-module. We regard $\sigma$ as a representation of $U$ in the
obvious way (that is, we let $U_v$ act trivially if $v\nmid p$). Fix
also a character $\psi:F^\times\backslash(\mathbb{A}^f_F)^\times\to A^\times$
such that for any finite place $v$ of $F$,
$\sigma|_{U_v\cap\bigO_{F_v}^\times}$ is multiplication by
$\psi^{-1}$. Then we can think of $W_\sigma$ as a
$U(\mathbb{A}_F^f)^\times$-module by letting $(\mathbb{A}_F^f)^\times$
act via $\psi^{-1}$.

Let $S_{\sigma,\psi}(U,A)$ denote the set of continuous
functions $$f:D^\times\backslash(D\otimes_F\mathbb{A}_F^f)^\times\to
W_\sigma$$ such that for all $g\in (D\otimes_F\mathbb{A}_F^f)^\times$
we have $$f(gu)=\sigma(u)^{-1}f(g)\text{ for all }u\in
U,$$ $$f(gz)=\psi(z)f(g)\text{ for all
}z\in(\mathbb{A}_F^f)^\times.$$We can write
$(D\otimes_F\mathbb{A}_F^f)^\times=\coprod_{i\in I}D^\times
t_iU(\mathbb{A}_F^f)^\times$ for some finite index set $I$ and some
$t_i\in(D\otimes_F\mathbb{A}_F^f)^\times$. Then we
have $$S_{\sigma,\psi}(U,A)\isoto\oplus_{i\in
  I}W_\sigma^{(U(\mathbb{A}_F^f)^\times\cap t_i^{-1}D^\times
  t_i)/F^\times},$$the isomorphism being given by the direct sum of
the maps $f\mapsto f(t_i)$. From now on we make the following
assumption:$$\text{For all
}t\in(D\otimes_F\mathbb{A}_F^f)^\times\text{ the group
}(U(\mathbb{A}_F^f)^\times\cap t^{-1}D^\times t)/F^\times\text{ is
  trivial.}$$ Without changing $U_\p$, one can always replace $U$ by a
subgroup satisfying the above assumptions (\emph{cf.} section 3.1.1 of
\cite{kis07}). Under these assumptions $S_{\sigma,\psi}(U,A)$ is a
finite free $A$-module, and the functor $W_\sigma\mapsto
S_{\sigma,\psi}(U,A)$ is exact in $W_\sigma$.

We now define some Hecke algebras. Let $S$ be a set of finite places
containing $\Sigma$, $\p$, and the primes $v$ of $F$ such that $U_v
\neq (\bigO_D)_v^\times$.  Let
$\T^{\operatorname{univ}}_{S,A}=A[T_v,S_v]_{v\notin S}$ be the
commutative polynomial ring in the formal variables $T_v$,
$S_v$. Consider the left action of $(D\otimes_F\mathbb{A}_F^f)^\times$
on $W_\sigma$-valued functions on $(D\otimes_F\mathbb{A}_F^f)^\times$
given by $(gf)(z)=f(zg)$. Then we make $S_{\sigma,\psi}(U,A)$ a
$\T^{\operatorname{univ}}_{S,A}$-module by letting $S_v$ act via the
double coset $U\bigl(
\begin{smallmatrix}
    \pi_{v}&0\\0&\pi_{v}
\end{smallmatrix}
\bigr)U$ and $T_v$ via $U\bigl(
\begin{smallmatrix}
    \pi_{v}&0\\0&1
\end{smallmatrix}
\bigr)U$. These are independent of the choices of $\pi_{v}$. We will write $\T_{\sigma,\psi}(U,A)$ or $\T_{\sigma,\psi}(U)$ for the image of $\T^{\operatorname{univ}}_{S,A}$ in $\End S_{\sigma,\psi}(U,A)$.

Note that if $\sigma$ is trivial, then we may also define Hecke
operators at $\p$. We let $U_{\pi_\p}$ be the Hecke operator given by the
double coset $U\bigl(
\begin{smallmatrix}
    \pi_{\p}&0\\0& 1
\end{smallmatrix}
\bigr)U$ and let $V_{\pi_\p}$ be given by $U\bigl(
\begin{smallmatrix}
    1&0\\0&\pi_\p
\end{smallmatrix}
\bigr)U$. Note that these may depend on the choice of $\pi_\p$.

Let $\mathfrak{m}$ be a maximal ideal of
$\T^{\operatorname{univ}}_{S,A}$. We say that $\mathfrak{m}$ is in the
support of $(\sigma,\psi)$ if $S_{\sigma,\psi}(U,A)_\mathfrak{m}\neq
0$. Now let $\bigO$ be the ring of integers in $\Qpbar$, with residue
field $\F=\Fpbar$, and suppose that $A=\bigO$ in the above discussion
and that $\sigma$ has open kernel. Consider a maximal ideal
$\mathfrak{m}\subset\T^{\operatorname{univ}}_{S,\bigO}$ which is
induced by a maximal ideal of $\T_{\sigma,\psi}(U,\bigO)$. Then (see
for example section 3.4.1 of \cite{kis04}) there
is a semisimple Galois representation
$\overline{\rho}_\mathfrak{m}:G_F\to\GL_2(\F)$ associated to
$\mathfrak{m}$ which is characterised up to conjugacy by the
property that if $v\notin S$ then $\rhobar_\mathfrak{m}|_{G_{F_v}}$ is
unramified, and if $\Frob_v$ is an arithmetic Frobenius
at $v$ then the trace of $\overline{\rho}_\mathfrak{m}(\Frob_v)$ is
the image of $T_v$ in $\F$.

We are now in a position to define what it means for a
representation to be modular of some weight. Let $F_{\p}$ have ring
of integers $\bigO_{F_\p}$, and let $\sigma$ be an irreducible
$\F$-representation of $\GL_2(\F_{p})$, so $\sigma$ is isomorphic to
$\sigma_{m,n} := \det^m\otimes\Sym^n\F^2$ for some $0\leq m< p-1$,
$0\leq n\leq p-1$. Throughout the paper we allow $m,n$ to vary over
these ranges.   We also denote by $\sigma$ the representation of
$\GL_2(\bigO_{F_{\p}})$ induced by the surjection
$\bigO_{F_{\p}}\onto \F_{p}$.
\begin{defn}
    We say that $\overline{\rho}$ is modular of weight $\sigma$ if
        for some $D$, $S$, $U$, $\psi$, and $\mathfrak{m}$ as above,
        with $U_\p=\GL_2(\bigO_{F_\p})$, we have $S_{\sigma,\psi}(U,\F)_\mathfrak{m}\neq 0$ and $\overline{\rho}_\mathfrak{m}\cong\overline{\rho}$.
\end{defn}
Assume from now on that $\rhobar$ is modular of some weight, and fix
$D$, $S$, $U$, $\psi$, $\mathfrak{m}$ as in the definition. Write
$W(\rhobar)$ for the set of weights $\sigma$ for which $\rhobar$ is
modular of weight $\sigma$.

One can gain information about the weights associated to a particular
Galois representation by considering lifts to characteristic zero. The
key is the following basic lemma.
\begin{lem}
  \label{432}Let
  $\psi:F^{\times}\backslash(\A_{F}^f)^{\times}\to\bigO^{\times}$ be a
  continuous character, and write $\overline{\psi}$ for the composite
  of $\psi$ with the projection $\bigO^{\times}\to \F^{\times}$. Fix a
  representation $\sigma$ of $U_\p$ on a finite free $\bigO$-module
  $W_{\sigma}$, and an irreducible representation $\sigma'$ of $U_\p$ on a
  finite free $\F$-module $W_{\sigma'}$. Suppose that we have
  $\sigma|_{U_v\cap\bigO_{F_v}^\times}=\psi^{-1}|_{U_v\cap\bigO_{F_v}^\times}$
  and
  $\sigma'|_{U_v\cap\bigO_{F_v}^\times}=\overline{\psi}^{-1}|_{U_v\cap\bigO_{F_v}^\times}$
  for all finite places $v$.

    Let $\mathfrak{m}$ be a maximal ideal of $\mathbb{T}_{S,\bigO}^{\univ}$.

    Suppose that $W_{\sigma'}$ occurs as a $U_\p$-module subquotient of ${W}_{\overline{\sigma}}:=W_\sigma\otimes\F$. If $\mathfrak{m}$ is in the support of $(\sigma',\overline{\psi})$, then $\mathfrak{m}$ is in the support of $(\sigma,\psi)$.

    Conversely, if $\mathfrak{m}$ is in the support of
        $(\sigma,\psi)$, then $\mathfrak{m}$ is in the support of
        $(\sigma',\overline{\psi})$ for some irreducible $U_\p$-module
        subquotient $W_{\sigma'}$ of ${W}_{\overline{\sigma}}$.
\end{lem}
\begin{proof}
  The first part is proved just as in Lemma 3.1.4 of \cite{kis04}, and
  the second part follows from Proposition 1.2.3 of \cite{as862}.
\end{proof}

\section{Tame lifts}\label{sec:tame lifts} We recall some group-theoretic results from section 3 of \cite{cdt}. First, recall the irreducible finite-dimensional representations of $\GL_2(\F_p)$ over $\overline{\Q}_p$. Once one fixes an embedding $\F_{p^2}\into M_2(\F_p)$, any such representation is equivalent to one in the following list:
\begin{itemize}
    \item For any character $\chi:\F_p^\times\to\overline{\Q}_p^\times$, the representation $\chi\circ\det$.
    \item For any $\chi:\F_p^\times\to\overline{\Q}_p^\times$, the representation $\operatorname{sp}_\chi=\operatorname{sp}\otimes(\chi\circ\det)$, where $\operatorname{sp}$ is the representation of $\GL_2(\F_p)$ on the space of functions $\mathbb{P}^1(\F_p)\to\overline{\Q}_p$ with average value zero.
    \item For any pair $\chi_1\neq\chi_2:\F_p^\times\to\overline{\Q}_p^\times$, the representation $$I(\chi_1,\chi_2)=\Ind_{B(\F_p)}^{\GL_2(\F_p)}(\chi_1\otimes\chi_2),$$ where $B(\F_p)$ is the Borel subgroup of upper-triangular matrices in $\GL_2(\F_p)$, and $\chi_1\otimes\chi_2$ is the character $$ \left(
    \begin{array}{cc}
        a & b \\
        0 & d \\
    \end{array}
    \right)\mapsto\chi_1(a)\chi_2(d). $$
    \item For any character $\chi:\F_{p^2}^\times\to\overline{\Q}_p^\times$ with $\chi\neq\chi^p$, the cuspidal representation $\Theta(\chi)$ characterised by $$\Theta(\chi)\otimes\operatorname{sp}\cong\Ind_{\F_{p^2}^\times}^{\GL_2(\F_p)}\chi.$$
\end{itemize}
We now recall the reductions mod $p$ of these representations. Let $\sigma_{m,n}$ be the irreducible $\overline{\F}_p$-representation $\det^m\otimes\Sym^{n}\F^{2}$, with $0\leq m<p-1$, $0\leq n\leq p-1$. Then we have:
\begin{lem}
    \label{441}Let $L$ be a finite free $\bigO$-module with an action
    of $\GL_2(\F_p)$ such that $V=L\otimes_\bigO\overline{\Q}_p$ is
    irreducible. Let $\tilde{a}$ denote the Teichm\"{u}ller lift of $a$.
    \begin{enumerate}
        \item If $V\cong\chi\circ\det$ with $\chi(a)=\tilde{a}^m$, then $L\otimes_\bigO \F\cong\sigma_{m,0}$.
        \item If $V\cong\operatorname{sp}_{\chi}$ with $\chi(a)=\tilde{a}^m$, then $L\otimes_\bigO \F\cong\sigma_{m,p-1}$.
        \item If $V\cong I(\chi_1,\chi_2)$ with $\chi_i(a)=\tilde{a}^{m_i}$ for distinct $m_i\in\Z/(p-1)\Z$, then $L\otimes_\bigO \F$ has two Jordan-H\"{o}lder subquotients: $\sigma_{m_2,\{m_1-m_2\}}$ and $\sigma_{m_1,\{m_2-m_1\}}$ where $0<\{m\}<p-1$ and $\{m\}\equiv m \text{ mod }p-1$.
        \item If $V\cong\Theta(\chi)$ with $\chi(c)=\tilde{c}^{i+(p+1)j}$ where $1\leq i\leq p$ and $j\in\Z/(p-1)\Z$, then $L\otimes_\bigO \F$ has two Jordan-H\"{o}lder subquotients: $\sigma_{1+j,i-2}$ and $\sigma_{i+j,p-1-i}$. Both occur unless $i=p$ (when only the first occurs), or $i=1$ (when only the second one occurs), and in either of these cases $L\otimes_\bigO \F\cong\sigma_{1+j,p-2}$ .
    \end{enumerate}
\end{lem}
\begin{proof}
    This is Lemma 3.1.1 of \cite{cdt}.
\end{proof}



In what follows, we will sometimes consider the above
representations as representations of $\GL_{2}(\bigO_{F_{\p}})$ via
the natural projection map.

We now recall some definitions relating to potentially semistable
lifts of particular type. We use the conventions of \cite{sav04}.

\begin{defn}\label{defn:barsotti tate lifts} Let $\tau$ be an inertial
  type. We say that a lift $\rho$ of $\rhobar|_{G_{F_{\p}}}$ is
 parallel potentially Barsotti-Tate (respectively parallel potentially semistable) of type $\tau$ if $\rho$ is potentially
  Barsotti-Tate (respectively potentially semistable with all Hodge-Tate
  weights equal to 0 or 1) with $\det\rho=\epsilon\psi$, where
  $\epsilon$ denotes the cyclotomic character and 
 $\psi$ is some  finite order character of order
  prime to $p$, and moreover the corresponding
  Weil-Deligne representation, when restricted to $I_{F_{\p}}$, is
  isomorphic to $\tau$. \end{defn}

Note that for a two-dimensional de Rham representation with all
Hodge-Tate weights equal to $0$ or $1$, the condition that all pairs of labeled
Hodge-Tate weights are $\{0,1\}$ is equivalent to the condition that 
the determinant is the product of the cyclotomic character, a finite
order character, and an unramified character; the condition of being
parallel is slightly stronger than this.

\begin{defn}
  \label{defn:ordinary}
  We say that a representation $\rho:G_{F_\p}\to\GL_2(\Qpbar)$ is
  ordinary if $\rho|_{I_{F_\p}}$ is an extension of a finite order
  character by a finite order character times the cyclotomic
  character.
We say that a
  representation $\rho:G_F\to\GL_2(\Qpbar)$ is ordinary if
  $\rho|_{G_{F_\p}}$ is ordinary.
\end{defn}
We now need some  special cases of the inertial local Langlands
correspondence of Henniart (see the appendix to \cite{bm}). If
$\chi_{1}\neq\chi_{2}: \F_{p}^{\times}\to\bigO^{\times}$, let
$\tau_{\chi_{1},\chi_{2}}$ be the inertial type
$\chi_{1}\oplus\chi_{2}$ (considered as a representation of
$I_{F_{\p}}$ via local class field theory). Then we let
$\sigma(\tau_{\chi_{1},\chi_{2}})$ be a representation on a finite
$\bigO$-module given by taking a lattice in $I(\chi_{1},\chi_{2})$. If
$\chi:\F_{p}^{\times}\to\bigO^{\times}$, we let
$\tau_{\chi}=\chi\oplus\chi$, and $\sigma(\tau_{\chi})$ be
$\chi\circ\det$.

If $\chi:\F_{p^2}^{\times} \to\bigO^{\times}$ with
$\chi\neq\chi^p$, we let $\tau_{\chi,\chi^p}=\chi\oplus\chi^p$
(again, regarded as a representation of $I_{F_\p}$ via local class
field theory), and we let $\sigma(\tau_{\chi,\chi^p})$ be a
representation on  a finite $\bigO$-module given by taking a lattice
in $\Theta(\chi)$.


\begin{lem}\label{lem:typesversusweights}Fix a type $\tau$ as above
  (i.e., $\tau=\tau_{\chi_{1},\chi_{2}}$, $\tau_{\chi}$, or
  $\tau_{\chi,\chi^p}$). Suppose that $\rhobar$ is modular of weight
  $\sigma$, and that $\sigma$ is a $\GL_{2}(\Fp)$-module subquotient
  of $\sigma(\tau)\otimes_{\bigO} \F$. Then $\rhobar$ lifts to a
  modular Galois representation which is parallel potentially
  Barsotti-Tate of type $\tau$ at $\p$. Similarly, if $\rhobar$ is
  modular of weight $\sigma_{m,p-1}$, then $\rhobar$ lifts to a
  modular Galois representation which is parallel potentially
  semistable of type $\tilde{\omega}^m\oplus\tilde{\omega}^m$ (see the
  beginning of section \ref{sec: necessity} for the definition of the
  fundamental character $\omega$).  Conversely, if $\rhobar$ lifts to
  a modular Galois representation which is parallel potentially
  Barsotti-Tate of type $\tau$ at $\p$, then $\rhobar$ is modular of
  weight $\sigma$ for some $\GL_{2}(\Fp)$-module subquotient $\sigma$
  of $\sigma(\tau)\otimes_{\bigO} \F$.

\end{lem}
\begin{proof}
  By Lemma \ref{432}, the Jacquet-Langlands
  correspondence, and the discussion in section 3.1.14 of \cite{kis04}, $\rhobar$ is modular of weight
  $\sigma$ for some $\GL_{2}(\Fp)$-module subquotient $\sigma$
  of $\sigma(\tau)\otimes_{\bigO} \F$ if and only if there is an
  automorphic representation $\pi$ of $\GL_2(\A_F)$ corresponding to a
  Hilbert modular form of parallel weight 2, such that
  \begin{itemize}
  \item the Galois representation $\rho_\pi:G_F\to\GL_2(\Qpbar)$
    associated to $\pi$ lifts
    $\rhobar$, and
  \item $\pi_\p$ contains $\sigma(\tau)^*$ as a $\GL_2(\bigO_\p)$-representation.
  \end{itemize}
Since $\pi$ corresponds to a Hilbert modular form of parallel weight
2, $\rho_\pi|_{G_{F_\p}}$ is potentially semistable with Hodge-Tate
weights $0$ and $1$, and is potentially Barsotti-Tate provided that
$\pi_\p$ is not a twist of the Steinberg representation. By the compatibility of the local
and global Langlands correspondences at places dividing $p$ (see
\cite{kis06}), and the results of sections A2 and A3 of Henniart's
appendix to \cite{bm}, it follows that the conditions above on $\pi$
are equivalent to
\begin{itemize}
 \item the Galois representation $\rho_\pi:G_F\to\GL_2(\Qpbar)$
   associated to $\pi$ lifts
    $\rhobar$, and
  \item $\rho_\pi|_{G_{F_\p}}$ has type $\tau$.
\end{itemize}
Furthermore, in each of these cases it follows from the discussion in
section A2 of Henniart's appendix to \cite{bm} that $\pi_\p$ is not a twist of the Steinberg
representation.

This proves everything apart from the assertion that if $\rhobar$ is modular of weight
  $\sigma_{m,p-1}$, then $\rhobar$ lifts to a modular Galois
  representation which is parallel potentially semistable of type
  $\tilde{\omega}^m\oplus\tilde{\omega}^m$. This may be proved by a
  very similar argument; in this case, the argument above provides an automorphic representation $\pi$ of $\GL_2(\A_F)$ corresponding to a
  Hilbert modular form of parallel weight 2, such that \begin{itemize}
 \item the Galois representation $\rho_\pi:G_F\to\GL_2(\Qpbar)$
   associated to $\pi$ lifts
    $\rhobar$, and
 \item $\pi_\p$ contains $\operatorname{sp}_{\tilde{\omega}^m\circ\det}^*$ as a $\GL_2(\bigO_\p)$-representation.
\end{itemize}
It then follows from section A2  of Henniart's appendix to \cite{bm} that $\pi_\p$ is a twist of the Steinberg
representation, from which the result follows easily.
\end{proof}

\section{Weight cycling}\label{sec: weight cycling}We now explain the weight cycling argument due
to Kevin Buzzard which proves modularity in an additional weight in
the non-ordinary case. There is an exposition of this argument in
section 5 of \cite{taymero} in the case that $p$ splits completely in
$F$, and the argument goes over essentially unchanged in our
setting. Since our notation and assumptions differ from those of
\cite{taymero}, we sketch a proof here.

Fix a character
$\psi:F^\times\backslash(\A_F^f)^\times\to\bigO^\times$. If $A$ is an
$\bigO$-module we will regard $\psi$ as an $A$-valued character via the
structure homomorphism. Suppose that $A$ is a field and that $\sigma^\vee$ denotes the dual of
$\sigma$. Then (\emph{cf.} the discussion on page 742 of
\cite{taymero}, recalling that by assumption we have
$(U(\A_F^f)^\times \cap t^{-1}D^\times t)/F^\times=1$ for all
$t\in(D\otimes_F\A_F^f)^\times$) there is a perfect pairing
$$\langle\cdot,\cdot\rangle:S_{\sigma,\psi}(U,A)\times
S_{\sigma^\vee,\psi^{-1}}(U,A)\to A$$ given by $$\langle
  f_1,f_2\rangle=\sum_i \langle f_1(t_i),f_2(t_i)\rangle$$
  where $$(D\otimes_F\A_F^f)^\times=\coprod_i D^\times
  t_iU(\A_F^f)^\times$$ and the pairing between $f_1(t_i)$ and $f_2(t_i)$
  is the usual pairing between a representation and its dual. A
  standard calculation shows that under this pairing the adjoint of
  $S_x$ is $S_x^{-1}$, the adjoint of $T_x$ is $T_xS_x^{-1}$, and when
  it is defined the adjoint of $U_{\pi_\p}$ is $V_{\pi_\p}S_\p^{-1}$.

Let $$U_0=\prod_{v\nmid p}U_v\times I_\p$$ where $I_\p$ is the
Iwahori subgroup of $\GL_2(\bigO_{F_\p})$ consisting of matrices
which are upper-triangular modulo $\p$, and let $$U_1=\prod_{v\nmid
p}U_v\times I^1_\p$$ where $I^1_\p$ is the subgroup of $I_\p$ whose
entries are congruent to                $\bigl(
\begin{smallmatrix}
    *&*\\0&1
\end{smallmatrix}
\bigr)$modulo $\p$. Let $\sigma=1$, the trivial representation, so
that the operators $U_{\pi_\p}$ and $V_{\pi_\p}$ are defined on
$S_{1,\psi}(U_1,A)$ and $S_{1,\psi^{-1}}(U_1,A)$ for any
$\bigO$-algebra $A$. Let $$\delta^{n}   :\bigl(
    \begin{smallmatrix}
        a&b\\0&d
    \end{smallmatrix}
    \bigr)\mapsto d^{n},$$    an $\F^\times$-character of the
        standard Borel subgroup $B(\Fp)$ of $\GL_{2}(\Fp)$. Then there
        is a natural embedding $$S_{\delta^n,\psi}(U_0,\F)\into
        S_{1,\psi}(U_1,\F)$$which is equivariant for the actions of
        $\T_{S,\F}^{\univ}$, and the image of
        $S_{\delta^n,\psi}(U_0,\F)$ is stable under the actions of
        $U_{\pi_\p}$ and $V_{\pi_\p}$. We use this action as the
        definition of $U_{\pi_\p}$ and $V_{\pi_\p}$ on $S_{\delta^n,\psi}(U_0,\F)$.

There is also a natural isomorphism $$S_{\delta^n,\psi}(U_0,\F)\cong
S_{\Ind(\delta^n),\psi}(U,\F)$$ where $\Ind(\delta^n)$ is obtained as
the induction from $B(\F_p)$ to $\GL_2(\F_p)$ of $\delta^n$. We use
this isomorphism to define actions of $U_{\pi_\p}$ and $V_{\pi_\p}$ on $S_{\Ind(\delta^n),\psi}(U,\F)$.
We can
think of this induction as being the functions $$\theta:\GL_2(\Fp)\to\F$$
with the property that for all $b\in B(\Fp)$,
$g\in\GL_2(\Fp)$, $$\theta(bg)=\delta^n(b)\theta(g).$$ The action of
$\GL_2(\Fp)$ is by $$(g\theta)(x)=\theta(xg).$$ This
isomorphism identifies $f\in S_{\delta^n,\psi}(U_0,\F)$ with $F\in
S_{\Ind(\delta^n),\psi}(U,\F)$ where $$f(x)=F(x)(1)$$
and $$F(x)(g)=f(xg^{-1}).$$

Now, we have a short exact
sequence $$0\to\sigma_{0,n}\to\Ind(\delta^n)\to\sigma_{n,p-1-n}\to
0 $$of $\GL_2(\Fp)$-modules and thus a short exact sequence $$0\to
S_{\sigma_{0,n},\psi}(U,\F)\stackrel{\alpha}{\to} S_{\Ind(\delta^n),\psi}(U,\F)\stackrel{\beta}{\to}
S_{\sigma_{n,p-1-n},\psi}(U,\F)\to 0$$and, localising at
$\mathfrak{m}$, a short exact sequence $$0\to
S_{\sigma_{0,n},\psi}(U,\F)_\mathfrak{m}\stackrel{\alpha}{\to} S_{\Ind(\delta^n),\psi}(U,\F)_\mathfrak{m}\stackrel{\beta}{\to}
S_{\sigma_{n,p-1-n},\psi}(U,\F)_\mathfrak{m}\to 0.$$

\begin{prop}
  \label{prop:weightcycling}
 If $n<p-1$ and $S_{\sigma_{0,n},\psi}(U,\F)_\mathfrak{m}=0$ then the map $V_{\pi_\p}:
 S_{\Ind(\delta^n),\psi}(U,\F)_\mathfrak{m}\to
 S_{\Ind(\delta^n),\psi}(U,\F)_\mathfrak{m}$ is an isomorphism.
\end{prop}
\begin{proof}
  Under the assumption that
  $S_{\sigma_{0,n},\psi}(U,\F)_\mathfrak{m}=0$, we have an
  isomorphism $$S_{\Ind(\delta^n),\psi}(U,\F)_\mathfrak{m}\stackrel{\beta}{\to}
  S_{\sigma_{n,p-1-n},\psi}(U,\F)_\mathfrak{m}.$$ We claim that
  there is an injection $$\kappa:
  S_{\sigma_{n,p-1-n},\psi}(U,\F)_\mathfrak{m}\to
  S_{\Ind(\delta^n),\psi}(U,\F)_\mathfrak{m}$$ such that
  $\kappa\circ\beta=V_{\pi_\p}$. This would clearly establish the
  result. Of course, it is enough to construct an injection $$\kappa:
  S_{\sigma_{n,p-1-n},\psi}(U,\F)\to S_{\Ind(\delta^n),\psi}(U,\F)$$
  which commutes with the action of $\mathbb{T}_{S,\F}^{\univ}$, and
  satisfies $\kappa\circ\beta=V_{\pi_\p}$.


The verification that this
is possible is exactly as in \cite{taymero}.   \end{proof}

\begin{prop}
  \label{prop:weightcycling in weight 0}
 If  $S_{\sigma_{0,p-1},\psi}(U,\F)_\mathfrak{m}=0$ then the map $U_{\pi_\p}:
 S_{\sigma_{0,0},\psi}(U,\F)_\mathfrak{m}\to
 S_{\sigma_{0,0},\psi}(U,\F)_\mathfrak{m}$ is an isomorphism.
\end{prop}
\begin{proof}
  Under the assumption that
  $S_{\sigma_{0,p-1},\psi}(U,\F)_\mathfrak{m}=0$, we have an isomorphism $$S_{\Ind(1),\psi}(U,\F)_\mathfrak{m}\stackrel{\beta}{\to}
S_{\sigma_{0,0},\psi}(U,\F)_{\mathfrak{m}}.$$ We claim that
there is an injection $$\kappa:
S_{\sigma_{0,0},\psi}(U,\F)\to
S_{\Ind(1),\psi}(U,\F)$$ which commutes with the action of
$\mathbb{T}_{S,\F}^{\univ}$, and satisfies
$\beta\circ\kappa=U_{\pi_\p}$.


Again, the verification that this
is possible is exactly as in \cite{taymero}. \end{proof}


Now, there is a maximal ideal
$\m^*$ of $\mathbb{T}_{S,\bigO}^{\univ}$ with the property that for
each $x\notin S$, $T_x-\alpha\in\m$ if and only if $T_x
S_x^{-1}-\alpha\in\m^*$, and $S_x-\beta\in\m$ if and only if
$S_x^{-1}-\beta\in\m^*$. Thus
 $$\rhobar_{\mathfrak{m}^*}\cong\rhobar_{\mathfrak{m}}^\vee(1).$$
Then from the duality explained above between $S_{\sigma,\psi}$ and
$S_{\sigma^\vee,\psi^{-1}}$, we obtain

\begin{cor}
  \label{cor:weightcyclingdualversion}If $n<p-1$ and $S_{\sigma_{0,n},\psi}(U,\F)_\mathfrak{m}=0$ then the map $$U_{\pi_\p}:
 S_{\Ind(\delta^{-n}),\psi^{-1}}(U,\F)_{\m^*}\to
 S_{\Ind(\delta^{-n}),\psi^{-1}}(U,\F)_{\m^*}$$ is an isomorphism.

\end{cor}
\begin{proof}
  This follows at once from Proposition \ref{prop:weightcycling}.
\end{proof}

%
%

\begin{prop}\label{prop: non-ordinary global  lift gives both weights via cycling}
Suppose that $0\leq n\leq p-1$. If $\rhobar$ has a lift to a modular representation
$\rho:G_{F}\to\GL_{2}(\Qpbar)$ which is parallel potentially Barsotti-Tate of
type $\widetilde{\omega}^{m+n}\oplus\widetilde{\omega}^{m}$ and is not
 ordinary, then $\rhobar$ is modular of weight
$\sigma_{m,n}$ and of weight $\sigma_{m+n,p-1-n}$.
\end{prop}
\begin{proof}We know from Lemma \ref{lem:typesversusweights} that
  $\rhobar$ is modular of weight at least one of $\sigma_{m+n,p-1-n}$
  and $\sigma_{m,n}$. Suppose for the sake of contradiction (and
  without loss of generality) that
  $\rhobar$ is modular of weight $\sigma_{m+n,p-1-n}$ but not weight
  $\sigma_{m,n}$. Twisting, we may without loss of generality assume
  that $m=0$.

Take $D$, $S$, $U$, $\psi$ and $\m$ as in section \ref{sec: notation},
chosen so that if $n\neq p-1$ there is an eigenform in
$S_{\Ind(\tilde{\delta}^n),\psi}(U,\Qpbar)$  corresponding to $\rho$,
and if $n=p-1$ there is such a form in $S_{1,\psi}(U,\Qpbar)$. Note that
by local-global compatibility and standard properties of the local
Langlands correspondence, the assumption that $\rho$ is not ordinary
shows that $U_\p$ has a non-unit eigenvalue on
$S_{\Ind(\tilde{\delta}^n),\psi}(U,\Qpbar)$ (respectively
$S_{1,\psi}(U,\Qpbar)$). [One needs to check that if $U_\p$ has only
unit eigenvalues then $\rho$ is necessarily ordinary. To do this, one
uses local-global compatibility to show that the action of Frobenius
on the associated weakly admissible representation has a unit
eigenvalue; the corresponding eigenspace is then a
sub-weakly-admissible module, from which ordinarity follows. See
Proposition 5.3.1 of \cite{ger} for a much more general result.]

Suppose first that $n=p-1$. Then by Proposition
\ref{prop:weightcycling in weight 0}, $U_{\pi_\p}$ is an isomorphism on
$S_{1,\psi}(U,\F)$; but this is a contradiction.

Suppose now that $n<p-1$. Then by Corollary
\ref{cor:weightcyclingdualversion}, $U_{\pi_\p}$ is an isomorphism on
$S_{\Ind(\delta^{-n}),\psi^{-1}}(U,\F)_{\m^*}$. Again, this
is a contradiction, as by the above duality there is an eigenform in
$S_{\Ind(\tilde{\delta}^{-n}),\psi^{-1}}(U,\Qpbar)_{\m^*}$
corresponding to the non-ordinary representation $\rho^\vee(1)$.
\end{proof}

\section{Necessary conditions}\label{sec: necessity}

Suppose that $K$ is a finite extension of $\Qp$, with residue field
$k$.  Let $S_{k}=\{\tau:k\into\Fpbar\}$. For each $\tau\in S_{k}$ we
define the fundamental character $\omega_{K,\tau}$ corresponding to
$\tau$ to be the composite $$\xymatrix{I_{K}\ar[r]^{\sim} &
  \bigO_{K}^{\times}\ar[r] & k^{\times}\ar[r]^{\tau} &
  \Fpbar^{\times},}$$ where the first map is the isomorphism given by
local class field theory, normalised so that a uniformiser corresponds
to geometric Frobenius. We will generally suppress the subscript $K$
and write simply $\omega_{\tau}$.  If $\chi$ is a character of $G_{K}$ or
$I_{K}$, we denote its reduction mod $p$ by $\overline{\chi}$.

If $K$ is totally ramified (e.g. if $K=F_{\p}$), we let $\omega$ be
the unique fundamental character. Note that $\omega^{e(K/\Qp)}$ is the
(restriction to $I_{K}$ of the) mod $p$ cyclotomic character, e.g. by Lemma
  \ref{lem:existence of crystalline chars} below.  Let $\sigma_{1}$,
  $\sigma_{2}$ denote the two embeddings of the quadratic extension of
  $k$ into $\Fpbar$, and let $\omega_{\sigma_{1}}$,
  $\omega_{\sigma_{2}}$ denote the two corresponding fundamental
  characters of $I_{K}$.  Write $\om$, $\om_{\sigma_1}$, and $\om_{\sigma_2}$
  for the Teichm\"uller lifts of $\omega$, $\omega_{\sigma_1}$, and
  $\omega_{\sigma_2}$ respectively.

If $\rhobar|_{G_{F_{\p}}}$ is semisimple, then Schein defines a set of predicted weights for $\rhobar$ as follows.

\begin{defn}[\cite{scheinramified}]\label{def: the predicted weights}
    The set $W^{?}(\rhobar)$ is the set of weights $\sigma_{m,n}$ such that there exists $1\leq x\leq e$ with either $\rhobar|_{I_{F_{\p}}}\cong\omega_{\sigma_{1}}^{m+n+x}\omega_{\sigma_{2}}^{m+e-x}\oplus\omega_{\sigma_{1}}^{m+e-x}\omega_{\sigma_{2}}^{m+n+x}$ or $\rhobar|_{I_{F_{\p}}}\cong\omega^{m+n+x}\oplus\omega^{m+e-x}$.
\end{defn}

Let $W(\rhobar)$ be the set of weights $\sigma$ such that
$\rhobar$ is modular of weight $\sigma$.  Our aim in this section is
to prove that $W(\rhobar) \subset W^{?}(\rhobar)$.

\subsection{Breuil modules with descent data}
\label{sec:breuil-modules-with}

Let $k$ be a
finite extension of $\F_p$, define $K_0=W(k)[1/p]$, and let $K$ be a
finite totally ramified extension of $K_0$ of degree~$e'$.  Suppose that $L$ is a $p$-adic subfield of $K$ such
that $K/L$ is Galois and tamely ramified.  Assume further that there is a uniformiser $\pi$
of $\bigO_{K}$ such that $\pi^{e(K/L)}\in L$, where $e(K/L)$ is the
ramification
degree of $K/L$, and fix such a $\pi$.
Since $K/L$ is tamely ramified, 
the category of \emph{Breuil modules with
coefficients and descent data from $K$ to $L$} is defined as follows (see \cite{sav06}). Let $k_E$ be a finite extension of $\Fp$. The category $\BrMod_{\dd,L}$
consists of quadruples $(\mathcal{M},\Fil^1 \mathcal{M},\phi_{1},\{\widehat{g}\})$ where:

\begin{itemize}\item $\mathcal{M}$ is a finitely generated
  $(k\otimes_{\F_p} k_E)[u]/u^{e'p}$-module, free over $k[u]/u^{e'p}$.
\item $\Fil^1 \M$ is a $(k\otimes_{\F_p} k_E)[u]/u^{e'p}$-submodule of $\M$ containing $u^{e'}\M$.
\item $\phi_{1}:\Fil^1\M\to\M$ is $k_E$-linear and $\phi$-semilinear
  (where $\phi:k[u]/u^{e'p}\to k[u]/u^{e'p}$ is the $p$-th power map)
  with image generating $\M$ as a $(k\otimes_{\F_p} k_E)[u]/u^{e'p}$-module.
\item $\widehat{g}:\M\to\M$ are additive bijections for each
  $g\in\Gal(K/L)$, preserving $\Fil^1 \M$, commuting with the $\phi_1$-,
  and $k_E$-actions, and satisfying the following axioms:  if $a\in k\otimes_{\F_{p}} k_E$, $m\in\M$ then $\widehat{g}(au^{i}m)=g(a)((g(\pi)/\pi)^{i}\otimes 1)u^{i}\widehat{g}(m)$; for all
  $g_1,g_2\in\Gal(K/L)$ we have $\widehat{g}_1\circ
  \widehat{g}_2=\widehat{g_1\circ g}_2$; and the map 
  $\widehat{1}$ is the identity.\end{itemize}

The category $\BrMod_{\dd,L}$ is equivalent to the category of finite flat
group schemes over $\mathcal{O}_K$ together with a $k_E$-action and descent
data on the generic fibre from $K$ to $L$ (this equivalence depends on $\pi$).

In particular, to each object of $\BrMod_{\dd,L}$ we can associate a
$G_L$-representation that extends the $G_K$-action on the geometric
points of the corresponding group scheme.  
However, we choose in this paper to adopt the conventions of \cite{bm} and
\cite{sav04}, rather than those of \cite{bcdt}; thus rather than working
with the usual contravariant equivalence of categories, we work with a
covariant version of it, so that our formulae for generic fibres will differ
by duality and a twist from those following the conventions of \cite{bcdt}.
To be precise, we obtain the associated $G_{L}$-representation (which we will refer to as the generic fibre) of an object of $\BrMod_{\dd,L}$ via
the functor $T_{\st,2}^{L}$ defined immediately before \cite[Lem. 4.9]{sav04}.

Let $E$ be a finite extension of $\Qp$ with integers $\OO_E$, maximal
ideal $\m_E$, and residue field $k_E$.  Recall
from \cite[Sec. 2]{sav04} that the functor $D_{\st,2}^{K}$ is an equivalence of categories
  between the category of $E$-representations of $G_L$ which are semistable when
  restricted to $G_K$ and have Hodge-Tate weights in $\{0,1\}$, and
  the category of weakly admissible filtered $(\phi,N)$-modules $D$ with
  descent data and $E$-coefficients such that $\Fil^0 (K \otimes_{K_0}
  D) = K \otimes_{K_0} D$ and $\Fil^2 (K \otimes_{K_0}
  D) = 0$.

Suppose that $\rho$ is a representation
  in the source of $D_{\st,2}^{K}$.  In
  what follows we must assume that the reader has some familiarity with the notation
and terminology of \cite[Sec. 4]{sav04}, but we will endeavour to give
precise references to what we use.  

Write $S$ for the ring
  $S_{K,\OO_E}$ defined in the second paragraph of \emph{loc. cit}.
Then  \cite[Prop. 4.13]{sav04} gives an essentially surjective functor,
  again denoted $T_{\st,2}^{L}$, from the category
  \cite[Def. 4.1]{sav04} of \emph{strongly divisible modules $\M$ with
  $\OO_E$-coefficients and descent data} in $S[1/p] \otimes_{K_0
  \otimes E}
  D_{\st,2}^{K}(\rho)$ to the category of Galois-stable
  $\OO_E$-lattices in~$\rho$.  By \cite[Cor. 4.12]{sav04} this
  functor is compatible with reduction mod $\m_E$, so that
applying $T_{\st,2}^{L}$ to the object $(k \otimes_{\Fp}
k_E)[u]/(u^{e'p}) \otimes_{S/\m_E S} (\M/\m_E \M)$ of $\BrMod_{\dd,L}$ yields a
reduction mod $p$ of $\rho$.

To simplify notation, for the remainder of the paper we write simply
$\M/\m_E\M$ for the above reduction mod ~$\m_E$ of $\M$ in
$\BrMod_{\dd,L}$ (we will never mean the literal $S/\m_E S$-module).  Let $\ell$ denote the residue field of $L$.

\begin{lem}
  \label{lem:connected_breuil_module}
  Let $\overline{\chi} : \Gal(K/L) \rightarrow k_E^{\times}$ be a
  character, and let $c$ be an element of  $(\ell \otimes_{\Fp}
  k_E)^{\times}$.  Define 
  $\M(\overline{\chi},c)$, a free of rank one Breuil
  module with $k_E$-coefficients and descent data from $K$ to $L$, as
  follows: let $\M(\overline{\chi},c)$ have 
generator $v$, and set
$$ \Fil^1 \M(\overline{\chi},c) = \M(\overline{\chi},c),\qquad \phi_1(v) =
cv, \qquad \widehat{g}(v) = (1 \otimes \overline{\chi}(g))v$$
for $g \in \Gal(K/L)$.  Then $T_{\st,2}^{L}(\M(\overline{\chi},1)) =
\overline{\chi}$, and $T_{\st,2}^{L}(\M(\overline{\chi},c))$ is an
unramified
twist of $\overline{\chi}$.
\end{lem}

\begin{proof}
Let $\chi : \Gal(K/L) \rightarrow E^{\times}$ be the Teichm\"uller lift of
$\overline{\chi}$.  For $\tilde{c} \in (W(\ell) \otimes_{\Zp}
\OO_E)^{\times}$ lifting $c$,
define a weakly admissible filtered $(\phi,N)$-module with descent data
$D(\chi,\tilde{c})$ over $K_0 \otimes_{\Qp} E$ with generator $\mathbf{v}$,
$$\Fil^i (K \otimes_{K_0} D(\chi,\tilde{c})) =
\begin{cases}
  K \otimes_{K_0} D(\chi,\tilde{c}) & \text{if} \ i \le 1 \\
  0 & \text{if} \ i > 1
\end{cases}$$
and
$$ \phi(\mathbf{v})=p\tilde{c}\mathbf{v}, \qquad N=0, \qquad \widehat{g}\cdot \mathbf{v} = (1 \otimes \chi(g))\mathbf{v}.$$
Now one checks directly from the definition that there is a strongly divisible $\OO_E$-module $\M(\chi,\tilde{c})$ contained in
$S \otimes_{W(k)} D(\chi,\tilde{c})$ generated by $\mathbf{v}$ and satisfying
$$\Fil^1 \M(\chi,\tilde{c}) = \M(\chi,\tilde{c}),\qquad \phi_1(\mathbf{v}) = \tilde{c}\mathbf{v}, \qquad
\widehat{g}\cdot \mathbf{v} = (1 \otimes \chi(g))  \mathbf{v}.$$
In particular $\M(\chi,\tilde{c})/\m_E \M(\chi,\tilde{c}) = \M(\overline{\chi},c)$.
One checks (exactly as in \cite[Ex. 2.14]{sav04}) that
$D_{\st,2}^{K}(\chi) = D(\chi,1)$, and it is then standard that the
representation giving rise to $D(\chi,\tilde{c})$ is an unramified
twist of $\chi$.  The result now follows from the discussion of the
functor $T_{\st,2}^{L}$
immediately before the statement of the lemma.
\end{proof}

\subsection{Actual weights are predicted weights}
\label{sec:actual-weights-are}

In this section we make use of the results of \cite{sav06} to prove
results on the possible forms of Galois representations which are
modular of a specified weight. Suppose that $\rhobar$ is modular of
weight $\sigma_{m,n}$. Then by Lemma \ref{lem:typesversusweights},
$\rhobar |_{G_{F_{\p}}}$ has a parallel potentially semistable lift of type $\widetilde{\omega}^{m+n}\oplus\widetilde{\omega}^{m}$.

\begin{lem}\label{lem:determinant of modular of some weight}If $\rhobar$ is modular of weight $\sigma_{m,n}$, then $\det \rhobar|_{I_{F_{\p}}}=\omega^{2m+n+e}$.
\end{lem}
\begin{proof}As remarked above, $\rhobar |_{G_{F_{\p}}}$ has a parallel
  potentially semistable lift of type
  $\widetilde{\omega}^{m+n}\oplus\widetilde{\omega}^{m}$, say
  $\rho$. It suffices to check that $\det
  \rho |_{I_{F_{\p}}}=\varepsilon\widetilde{\omega}^{2m+n}$, where
  $\varepsilon$ is the $p$-adic cyclotomic character (recalling again
  that the reduction mod $p$ of $\varepsilon$ is $\omega^{e}$). This follows at once from Definition \ref{defn:barsotti tate lifts} and the results of section B.2 of \cite{cdt}.
\end{proof}

We begin by addressing the case when $\rhobar |_{G_{F_{\p}}}$ is reducible.

\begin{lem}\label{lem:bounding weights: reducible case}Suppose that $\rhobar$ is modular of weight $\sigma_{m,n}$, and that $\rhobar|_{G_{F_{\p}}}$ is reducible. Then $\rhobar|_{I_{F_{\p}}}\cong \left(\begin{matrix}\omega^{m+n+x} & *\\0& \omega^{m+e-x}

\end{matrix}\right)$ or     $\rhobar|_{I_{F_{\p}}}\cong \left(\begin{matrix}\omega^{m+e-x} & *\\0& \omega^{m+n+x}

    \end{matrix}\right)$ for some $1\leq x\leq e$.
\end{lem}
\begin{proof}
We  consider first the case $n=p-1$. By Lemma
\ref{lem:typesversusweights}, $\rhobar|_{G_{F_\p}}$ has a parallel potentially
semistable lift $\rho$ of type
$\widetilde{\omega}^{m}\oplus\widetilde{\omega}^{m}$.  

If $\rho$ is not
potentially crystalline then it is automatically ordinary in the sense of
Definition~\ref{defn:ordinary}.  (We briefly recall the argument:
twisting by a character, one reduces to the case where $\rho$ is semistable, i.e.,
where $m=0$.  Then it follows from the relation $N \varphi = p \varphi N$ that the kernel of $N$ on the weakly admissible filtered
$(\varphi,N)$-module associated to $\rho$ must be a free of rank one
submodule with Newton polygon of constant slope $-1$, hence
 itself is weakly admissible.  Therefore
$\rho$ has a subcharacter that is an unramified twist of the
cyclotomic character.)  So the result in this case follows (with $x=e$). 

 If the lift $\rho$ is in fact
potentially crystalline then without loss of generality we may twist
and  suppose that $m=0$, and thus that $\rhobar$ has a crystalline
lift of Hodge-Tate weights 0 and 1. Then $\rhobar|_{G_{F_\p}}$ is
flat, and by  Theorem 3.4.3 of \cite{MR0419467} it has a
subcharacter of the form $\omega^i$ for some $0 \le i \le e$.  Since
$m=0$ and $n=p-1$, this subcharacter is of the form $\omega^{m+n+x}$
with $1 \le x \le e$ whenever $i > 0$, and is of the form
$\omega^{m+e-x}$ whenever $i < e$.  The result now follows from Lemma \ref{lem:determinant of
  modular of some weight}.

For the remainder of the proof suppose that $n < p-1$.
Let $K$ be the splitting field of $u^{p^2-1} - \pi_{\p}$ over $F_{\p}$, and
let $\varpi$ be a choice of $\pi_{\p}^{1/(p^2-1)}$ in $K$.
Let  $k_2$ denote the residue field of $K$, and if $g \in
\Gal(K/F_{\p})$
define
$\overline{\eta}(g)$ to be the image of $g(\varpi)/\varpi$ in $k_2$.
The field $F_{\p}$ will play the role of $L$ in the argument.

    Set $$\tau =
    \widetilde{\omega}_{\sigma_1}^{m+n+1}\widetilde{\omega}_{\sigma_2}^{m-1}\oplus\widetilde{\omega}_{\sigma_1}^{m-1}\widetilde{\omega}_{\sigma_2}^{m+n+1}
    = \widetilde{\omega}_{\sigma_1}^{(n+2) + (p+1)(m-1)} \oplus
    \widetilde{\omega}_{\sigma_1}^{p((n+2) + (p+1)(m-1))}.$$  
By Lemma~\ref{441}(4) and the description of the
 inertial local Langlands correspondence that precedes
 Lemma~\ref{lem:typesversusweights}, 
the weight $\sigma_{m,n}$ is a
    Jordan-H\"older factor of $\sigma(\tau) \otimes_{\OO} \F$.
 It follows from Lemma \ref{lem:typesversusweights} that $\rhobar|_{G_{F_{\p}}}$
        has a lift to a parallel potentially Barsotti-Tate representation of
        type $\tau$,
        and we may suppose this lift to be valued in $\bigO_{E}$ for $E$ some
        finite extension of $\Qp$ with
residue field $k_E$ into which $k_2$ embeds, and with uniformiser
$\pi_E$.  The
lift becomes Barsotti-Tate over $K$, and the
        $\pi_{E}$-torsion in the corresponding $p$-divisible group
        gives rise to a finite flat $k_E$-module scheme $\G$ over $\bigO_{K}$ with
        descent data  to $F_{\p}$, with generic fibre
        $\rhobar|_{G_{F_{\p}}}$, such that the descent data to $F_{\p}$ is
$\omega_{\sigma_1}^{m+n+1}\omega_{\sigma_2}^{m-1}\oplus\omega_{\sigma_1}^{m-1}\omega_{\sigma_2}^{m+n+1}$.
We claim that this implies the lemma.

 Suppose that $\rhobar|_{G_{F_\p}}\cong
        \left(\begin{matrix}\psi_{1} & *\\0&
            \psi_{2} \end{matrix}\right)$. Then by a scheme-theoretic
        closure argument \cite[Lem. 4.1.3]{bcdt},~$\G$ must contain a
        finite flat subscheme $\G_1$ with descent data which has generic
        fibre $\psi_{1}$.  Twisting by a suitable power of
$\omega$, we may assume $m=0$.
By Theorem 3.5 of \cite{sav06}, we may write
        the  Breuil module $\mathcal{M}$ corresponding to $\G_1$
in the
        form
\begin{itemize}
\item $\mathcal{M}=((k_2 \otimes_{\Fp} k_E)[u]/u^{e(p^2-1)p})\cdot w$
\item $\Fil^1 \mathcal{M}=u^{r}\mathcal{M}$
\item $\phi_{1}(u^{r}w)=cw$ for some $c \in k_E^{\times}$
\item  $\widehat{g}(w)=(\overline{\eta}(g)^{\kappa} \otimes 1)w$ for $g
  \in \Gal(K/F_{\p})$.
\end{itemize}
Here $\kappa,r$ are integers with $\kappa \in [0,p^2-1)$ and $r \in
[0,e(p^2-1)]$ satisfying  $\kappa \equiv p(\kappa+r) \pmod{p^2-1}$ or
equivalently
$r \equiv (p-1)\kappa \pmod{p^2-1}$.  Since the descent data on $\G$
is $\omega_{\sigma_1}^{n+1-p} \oplus
\omega_{\sigma_2}^{n+1-p}$ it follows that
$\kappa$ must be congruent to one of $n+1-p$ or
$p(n+1-p) \pmod{p^2-1}$; this follows also from \cite[Cor 5.2]{geesavittquaternionalgebras}.
Hence $\kappa = n+p^2-p$ or $pn+(p-1)$.

In the first case we find $r \equiv (p-1)(n+2) \pmod{p^2-1}$, and
therefore $r = (p-1)(n+2) + y(p^2-1)$ for some $0 \le y < e$.
One checks that there is a map $f: \M(\omega^{n+y+1},c) \rightarrow \M$
mapping $v \mapsto u^{pr/(p-1)} w$, where $\M(\omega^{n+y+1},c)$ is
as defined in Lemma~\ref{lem:connected_breuil_module}.  The kernel of this map does not
contain
any free $k_2[u]/u^{e(p^2-1)p}$-submodules, and so by \cite[Prop
8.3]{SavittCompositio} the map $f$ induces an isomorphism on generic
fibres.
By Lemma~\ref{lem:connected_breuil_module} we deduce that $\psi_1 |_{I_{F_{\p}}} =
\omega^{n+x}$
with $x=y+1 \in [1,e]$.

In the second case we find $r \equiv (p-1)(p-n-1) \pmod{p^2-1}$, and
therefore $r = (p-1)(p-n-1) + y(p^2-1)$ for some $0 \le y < e$.
One checks that there is a map $\M(\omega^{y},c) \rightarrow \M$
mapping $v \mapsto u^{pr/(p-1)} w$.  As in the previous case this map
induces an isomorphism on generic fibres, and by Lemma~\ref{lem:connected_breuil_module} we deduce that $\psi_1 |_{I_{F_{\p}}} =
\omega^{e-x}$
with $x=e-y \in [1,e]$.

Now in either of the two cases the result follows from Lemma \ref{lem:determinant of modular of some
  weight}.\end{proof}

We now consider the irreducible case.

\begin{lem}\label{lem:bounding weights: niveau 2}Suppose that $\rhobar$ is modular of weight $\sigma_{m,n}$, and that $\rhobar|_{G_{F_{\p}}}$ is irreducible. Then  $\rhobar|_{I_{F_{\p}}}\cong\omega_{\sigma_{1}}^{m+n+x}\omega_{\sigma_{2}}^{m+e-x}\oplus\omega_{\sigma_{1}}^{m+e-x}\omega_{\sigma_{2}}^{m+n+x}$ for some $1\leq x\leq e$.

\end{lem}
\begin{proof}This may be proved in essentially the same way as Lemma
  \ref{lem:bounding weights: reducible case}. However, the result
  follows easily from the results of \cite{scheinramified}. Note that
  $\sigma_{m,n}$ is a Jordan-H\"{o}lder factor of
  $\det^m\otimes\Ind_{B(\Fp)}^{\GL_2(\Fp)}(\delta^{n})$. Then in the case $n\neq 0$, $p-1$
  the result follows at once from (the proof of) Proposition 3.3 of
  \cite{scheinramified}; while Schein works in the case of an
  indefinite quaternion algebra, his arguments are ultimately purely
  local, using Raynaud's classificiation of finite flat group schemes
  of type $(p,\dots,p)$.

In the case $n= 0$ or $p-1$ a very similar but rather easier analysis
applies. In this case, by Lemma \ref{lem:typesversusweights},
$\rhobar|_{G_{F_\p}}$ has a parallel potentially semistable lift of type
$\widetilde{\omega}^{m}\oplus\widetilde{\omega}^{m}$.  If it is not
potentially crystalline then it is automatically ordinary, a
contradiction. Thus the lift must in fact be potentially crystalline,
and after twisting, one needs only to consider the case $m=0$, and one
is reduced to determining the possible generic fibres of finite flat
group schemes over $F_{\p}$, which is immediate from Raynaud's
analysis (Theorem 3.4.3 of \cite{MR0419467}), together with Lemma \ref{lem:determinant of modular of some weight}.
\end{proof}

Putting Lemma \ref{lem:bounding
    weights: reducible case} and Lemma \ref{lem:bounding weights:
    niveau 2} together, we obtain the following.

\begin{cor}\label{cor:tame case, necessary conditions on galois rep}If
  $\rhobar|_{G_{F_{\p}}}$ is semisimple, then $W(\rhobar)\subset W^{?}(\rhobar)$.
\end{cor}

If $e \geq p-1$ then one can check that
  $W^{?}(\rhobar)$ is the whole set of weights $\sigma_{m,n}$ with
  $\det \rhobar|_{I_{F_{\p}}}=\omega^{2m+n+e}$, so in that case Corollary~\ref{cor:tame case, necessary conditions on galois rep} follows
  from Lemma~\ref{lem:determinant of modular of some weight} alone.

\section{Local lifts}\label{sec: crystalline chars}

If $\rhobar$ is to be modular of weight
$\sigma \in W^{?}(\rhobar)$, then Lemma~\ref{lem:typesversusweights} entails
that $\rhobar |_{G_{F_{\p}}}$ must have a parallel potentially Barsotti-Tate lift of
some particular type or types.  The
existence of certain of these local lifts will be a key ingredient
in the proof that $\rhobar$ is indeed modular of weight $\sigma$.  Our
aim in this section is to produce these lifts.

\subsection{The niveau $1$ case}
\label{sec:niveau-1-case}

Let $K$ be a finite extension of $\Qp$, and let $S_{K}$ denote the set of embeddings $\tau:K\into\Qpbar$.
\begin{defn}
    If $\rho$ is a $\Qpbar$-valued crystalline representation of $G_{K}$ and $\tau\in S_{K}$ we say that the Hodge-Tate weights of $\rho$ with respect to $\tau$ are the $i$ for which $$gr^{-i}_{\tau}(\rho):=gr^{-i}((\rho\otimes_{\Qp}\mathrm{B}_{\mathrm{dR}})^{G_{K}}\otimes_{\Qpbar\otimes_{\Qp}K,1\otimes\tau}\Qpbar)\neq 0,$$ counted with multiplicity $\dim gr^{-i}_{\tau}(\rho)$. We denote the multiset of Hodge-Tate weights of $\rho$ with respect to $\tau$ by $\HT_{\tau}(\rho)$; it has cardinality $\dim \rho$.
\end{defn}
If $\sigma\in S_{K}$, we let $\overline{\sigma}$ be the induced element of $S_{k}$.


\begin{lem}\label{lem:existence of crystalline chars}Let $A=\{a_{\tau}\}_{\tau\in S_{K}}$ be a set of integers. Then there is a crystalline character $\varepsilon^{K}_{A}$ of $G_{K}$ such that $HT_{\tau}(\varepsilon^K_{A})=a_{\tau}$ for all $\tau\in S_K$, and $\varepsilon^K_{A}$ is unique up to unramified twist. Furthermore, $\overline{\varepsilon^{K}_{A}}|_{I_{K}}=\prod_{\tau\in S_{k}}\omega_{\tau}^{b_{\tau}},$ where $$b_{\tau}=\sum_{\sigma\in S_{K}:\overline{\sigma}=\tau}a_{\sigma}.$$
\end{lem}

\begin{proof} We only sketch a proof; the
  full details will appear in \cite{geesavitt}. The existence of
  $\varepsilon^{K}_{A}$ is easy, as one has only to write down the
  corresponding weakly admissible filtered module. Uniqueness up to
  unramified twist is clear, because a crystalline character all of
  whose Hodge-Tate weights are $0$ is automatically
  unramified. Finally, to compute the reduction modulo $p$, it
  suffices to treat the case where all but one element of $A$ is $0$,
  and the remaining element is $1$, as the general case then follows
  by taking a product of such characters. In this case (with some effort) one can
  construct a strongly divisible module corresponding to a lattice in
  $\varepsilon^{K}_A$, and  from the reduction mod $p$ of that
  strongly divisible module one determines
  $\overline{\varepsilon^{K}_{A}}|_{I_{K}}$ by applying \cite[Cor. 2.7]{sav06}.
\end{proof}

\begin{lem}\label{lem:existence of lifts in the tame niveau 1 case}Suppose that $$\rhobar|_{I_{F_{\p}}}\cong \left(\begin{matrix}\omega^{m+n+x} & 0\\0& \omega^{m+e-x}

\end{matrix}\right)$$with $1\leq x\leq e$, and that $\rhobar
|_{G_{F_{\p}}}$ itself is decomposable (which is automatic if
$\omega^{m+n+x}\neq\omega^{m+e-x}$). Then $\rhobar |_{G_{F_{\p}}}$ has
a parallel potentially Barsotti-Tate lift $\rho$ of type
$\widetilde{\omega}^{m+n}\oplus\widetilde{\omega}^{m}$. If $x\neq e$
then there is a non-ordinary such lift; if $x=e$ then there is a
non-ordinary
such lift provided that $n+e > p-1$, unless $e \le p-1$ and $n=p-1$.
\end{lem}
\begin{proof}Let $A=\{a_{\tau}\}_{\tau\in S_{{F_\p}}}$ have exactly $x$
  elements equal to $1$, and the remaining $e-x$ elements equal to
  $0$. Let $B=\{1-a_{\tau}\}_{\tau\in S_{{F_\p}}}$. Then by Lemma
  \ref{lem:existence of crystalline chars} we may take $\rho$ to be
  given by an unramified twist of
  $\widetilde{\omega}^{m+n}\varepsilon^{{F_\p}}_{A}$ plus an unramified
  twist of $\widetilde{\omega}^{m}\varepsilon^{{F_\p}}_{B}$ (with the
  unramified twists chosen so that this is indeed a lift of
  $\rhobar$, and also so that $\rho$ is parallel). This lift is ordinary precisely if one of
  $\varepsilon^{{F_\p}}_{A}$ or $\varepsilon^{{F_\p}}_{B}$ is an unramified
  twist of the cyclotomic character, which occurs if and only if $x=e$.

Now suppose that $x=e$.  If $e > p-1$, then because $\omega^{p-1}=1$
we may instead
take $A$ to have exactly $x-(p-1)$ elements equal to $1$, and the rest
equal to $0$, and produce a non-ordinary lift with $B$ and $\rho$
defined as above.  If $e \le p-1$ but $n+e > p-1$,  we take $A$ to have
exactly
$n+e-(p-1)$ elements equal to $1$ and the rest equal to zero.  Set $B = \{1-a_{\tau}\}_{\tau \in
  S_{F_\p}}$
and take $\rho$ to be given by an unramified twist of $\om^{m}
\varepsilon_{A}^{{F_\p}}$ plus an unramified twist of $\om^{m+n} \varepsilon_{B}^{{F_\p}}$.
This is non-ordinary provided that $n \neq p-1$.
\end{proof}

\subsection{The niveau 2 case: some strongly divisible modules}\label{sec:the lifting result in the
  residually irred case}

In the remainder of the section we wish to prove the following.

\begin{lem}\label{lem:existence of lifts in the tame niveau 2 case} Suppose that
$\rhobar|_{G_{F_{\p}}}$ is irreducible and
$\rhobar|_{I_{F_{\p}}}\cong\omega_{\sigma_{1}}^{m+n+x}\omega_{\sigma_{2}}^{m+e-x}\oplus\omega_{\sigma_{1}}^{m+e-x}\omega_{\sigma_{2}}^{m+n+x}$
for some $1\leq x\leq e$. Then $\rhobar|_{G_{F_{\p}}}$ has a parallel potentially Barsotti-Tate lift of type $\widetilde{\omega}^{m+n}\oplus\widetilde{\omega}^{m}$.

\end{lem}

We prepare for the proof of Lemma~\ref{lem:existence of lifts in
  the tame niveau 2 case} by constructing certain strongly divisible modules.

Let $K$ be a totally ramified finite extension of $\Qp$ with
ramification index $e$ and residue field $k$, and
fix a uniformiser $\pi$ of $K$.  Let $K_2$ be the splitting field
of $u^{p^2-1} - \pi$, and write $e_2 =
(p^2-1)e$ and $k_2$ respectively
for the ramification index and residue field of $K_2$.  Similarly
write $e_1 = (p-1)e$.  Choose $\varpi \in K_2$ a uniformiser
with $\varpi^{p^2-1} = \pi$.  Let
$E(u)$ be an Eisenstein polynomial for $\varpi$, and write $E(u) = u^{e_2}
+ pF(u)$, so that $F(u)$ is a polynomial in $u^{p^2-1}$ over $W(k)$ whose constant
term is a unit.

If $g \in G_K$ write $\omt(g) = g\varpi/\varpi \in \mu_{p^2-1}(K_2)$
and write $\omega_2(g)$ for the image of $\omt(g)$ in
$k_2$.  Then $\om = \omt^{p+1} |_{I_K}$ and $\omega = \omega_2^{p+1}
|_{I_K}$,
but we will abuse notation and also write $\om, \omega$ for
$\omt^{p+1}, \omega_2^{p+1}$ respectively.   These characters may all equally well be regarded as functions on $\Gal(K_2/K)$.
Note that if $\sigma : k_2 \hookrightarrow \Fpbar$ then in the
notation of previous sections we have $\omega_{\sigma} = \sigma \circ
\omega_2 |_{I_{K_2}}$.

Let $E$ denote the coefficient field for our representations, with
integer
ring $\OO_E$ and maximal ideal $\m_E$.  Assume that $E$ is ramified
over $\Qp$
and
that $W(k_2)$ embeds into
$E$.  Let $k_E$ denote the residue field of $E$; to simplify notation
we fix an identification of $k_E$ with a subfield of $\Fpbar$.   Write $S = S_{K_2,\OO_E}$
(notation as in \cite[Sec. 4]{sav04}).  Recall that $\phi : S
\rightarrow S$ is the $W(k_2)$-semilinear, $\OO_E$-linear map sending $u
\mapsto u^{p}$.  The group $\Gal(K_2/K)$ acts $W(k_2)$-semilinearly on $S$
via $g \cdot u = (\omt(g) \otimes 1)u$.
Set $c = \frac{1}{p} \phi(E(u))\in S^{\times}$.

\begin{thm} \label{thm:strdiv}  Let $0 \le j \le e_1$ be an integer
  and set $J = (p+1)j$.
 There exists a strongly divisible $\OO_E$-module $\M = \M_j$ with tame descent data
from $K_2$ to $K$ and generators
 $g_1,g_2$ such that $\overline{\M} = \M/\m_E \M$ has the form
$$ \Fil^1 \M = \langle u^J \overline{g}_1 \ , \ u^{e_2-J}
\overline{g}_2 \rangle, $$
$$ \phi_1(u^J \overline{g}_1) = \overline{g}_2, \qquad
\phi_1(u^{e_2-J} \overline{g}_2) = \overline{g}_1 ,$$
$$ \widehat{g}(\overline{g}_1) = \overline{g}_1, \qquad
\widehat{g}(\overline{g}_2) = \omega^{n}(g) \overline{g}_2 $$
for $g \in \Gal(K_2/K)$, where $n$ is the least nonnegative residue of
$j$ modulo $p-1$.
\end{thm}

\begin{proof}
Let $\M$ be the
free $S$-module generated by $g_1,g_2$.  To give $\M$ the structure of
a strongly divisible module with descent data and coefficients, we
must specify $\Fil^1 \M$, $\phi$, $N$, and descent data and check
that these satisfy the conditions (1)-(12) of \cite[Def. 4.1]{sav04}.
This is what we now do.

Choose any $x_1,x_2 \in \m_E$ with $x_1 x_2 = p$ (this is where we use
the hypothesis that $E/\Qp$ is ramified)  and let $\Fil^1 \M$ be the
submodule of $\M$ generated by $$h_1 := u^J g_1 + x_1 g_2, \qquad h_2
:= -x_2 F(u) g_1 +
u^{e_2-J} g_2, \qquad (\Fil^1 S)\M.$$
We would like to define a map $\phi : \M \rightarrow \M$, semilinear
with respect to $\phi$ on $S$, so that $\phi_1 =
\frac{1}{p} \phi |_{\Fil^1 \M}$ is well-defined and satisfies
\begin{gather}
\label{phi-h1} \phi_1(u^J g_1 + x_1 g_2) = g_2 \\
\label{phi-h2} \phi_1(-x_2 F(u) g_1 + u^{e_2-J} g_2) = g_1.
\end{gather}
This entails $\phi_1(E(u) g_1) = \phi_1(u^{e_2-J} h_1 - x_1 h_2) =
u^{p(e_2-J)} g_2 - x_1 g_1$, suggesting that we should define
$$ \phi(g_1) =  c^{-1}(u^{p(e_2-J)} g_2 - x_1 g_1) $$
and similarly
$$ \phi(g_2) = c^{-1} (x_2 \phi(F(u)) g_2 + u^{pJ} g_1). $$
Extending this map $\phi$-semilinearly to all of $\M$, one checks that
equations \eqref{phi-h1} and \eqref{phi-h2} hold, so that $\phi(\Fil^1 \M)$ is contained in $p\M$
and generates it over $S$.

We will now check that this choice of $\Fil^1 \M$ satisfies condition (2) of \cite[Def. 4.1]{sav04}. Recall from  the discussion
 preceding \cite[Def. 4.1]{sav04} that each element of $S$ can be written
uniquely as $\sum_{i \ge 0} r_i(u) E(u)^{i} / i! $ where each
$r_i$ is a polynomial of degree less than $e_2$, and such an element
lies in $\Fil^1 S$ if and only if $r_0 = 0$.  We claim that each
coset in $\Fil^1 \M/(\Fil^1 S) \M$ has a representative $a h_1 + b h_2$ with
$a,b$
polynomials of degree less than
$e_2-J$ and $J$ respectively.  Indeed, given a coset $A h_1 + B h_2 +
(\Fil^1 S)\M$ with $A,B \in S$, we can alter the coset representative
as follows: write $B$ as the sum of a polynomial of degree less than
$e_2$ and an element of $\Fil^1 S$, and absorb $h_2$ times the latter into $(\Fil^1 S)
\M$; use the relation $u^J h_2 = E(u) g_2 - x_2F(u) h_1$ to eliminate
the terms in $B$ of degree at least $J$ (thus altering the coefficient of $h_1$); write the new coefficient of
$h_1$ as the sum of a polynomial of degree less than
$e_2$ and an element of $\Fil^1 S$, and absorb the latter into
$(\Fil^1 S) \M$; finally, use the relation $u^{e_2-J} h_1 = E(u) g_1 +
x_1 h_2$ to eliminate the terms of degree at least $e_2-J$ in the
coefficient
of $h_1$, noting that in this last step one does not re-introduce
terms of degree at least $u^J$ into the coefficient of $h_2$.

If $I$ is any ideal of $\OO_E$, we wish to see that $\Fil^1 \M \cap
I\M = I\Fil^1 \M$.   We have seen that an arbitrary element $m$ of $\Fil^1
\M$ has the form $m = a h_1 + b h_2 + s_1 g_1 + s_2 g_2$ with $s_1,s_2 \in
\Fil^1 S$ and $a,b$ polynomials of degree  less than
$e_2-J$ and $J$ respectively.  Suppose such an element lies in $I
\M$.   The coefficient of $g_2$ for this element is $x_1 a + u^{e_2-J}
b + s_2$.  This must lie in $IS$; but because an element $\sum_{i}
r_i(u) E(u)^i/i!$ (with $\deg(r_i) < e_2$ for all $i$) lies in $IS$ if
and only if all the coefficents of the polynomials $r_i$ lie in
$W(k_2) \otimes I$,
and because $x_1 a$, $u^{e_2-J} b$ have no terms in common of the same
degree, it
follows that that $s_2 \in
I(\Fil^1 S)$ and the coefficients of $b$ lie in $W(k_2) \otimes I$.  Then
$ah_1 + s_1 g_1$ still lies in $I\M$, and now we can see that
$s_1 \in I(\Fil^1 S)$ and the coefficients of $a$ lie in $W(k_2)
\otimes I$.
We conclude that $m \in I (\Fil^1 \M)$, as desired.

Next we turn to descent data.  If $g \in \Gal(K_1/K)$, set
$\widehat{g}(g_1) = g_1$ and $\widehat{g}(g_2) = \om^n(g) g_2$,
and extend $\widehat{g}$ to $\M$ semilinearly with respect to
the usual action of $g$ on $S$.
One sees that $\widehat{g}$ preserves $\Fil^1 \M$ (remember that
$F(u)$ is a polynomial in $u^{p^2-1}$ over $W(k)$) and commutes with $\phi$.
So, summarizing all our work so far, we have shown that the
tuple $(\M, \Fil^1 \M, \phi, \{ \widehat{g} \})$
satisfies all the axioms of a
strongly divisible $\OO_E$-module with tame descent data
\cite[Def. 4.1]{sav04} other than the axioms involving
the monodromy operator $N$.

Ignoring the action of $\OO_E$ and the descent data and regarding
$(\M, \Fil^1 \M, \phi)$ simply as a strongly divisible $\Zp$-module
over $K_2$,
it follows from \cite[Prop 5.1.3(1)]{bre00} that there exists
a \emph{unique} $W(k_2) \otimes \Zp$-endomorphism $N : \M \rightarrow \M$ satisfying
axioms (5)--(8) of \cite[Def. 4.1]{sav04}, except that we have
axiom (5) only with respect to $s \in S_{K_2,\Zp}$ until we know that
$N$ commutes with the action of~$\OO_E$.  For the latter, if $z \in
\OO_E^{\times}$ we observe that $z N z^{-1}$ satisfies the same list
of
axioms that determines $N$ uniquely, so $z$ and $N$ commute;
since $\OO_E^{\times}$ generates $\OO_E$ as a $\Zp$-module we conclude
that $N$ is an $\OO_E$-endomorphism.  To conclude that
$\M$ (with its associated structures) is a strongly divisible
$\OO_E$-module, the only thing left is to
confirm the remainder of axiom (12),
that $N$ commutes with $\widehat{g}$ for each $g \in \Gal(K_2/K)$; for
this, use the same argument as in the previous sentence, this time applied to
$\widehat{g} N \widehat{g}^{-1}$.

That $\overline{\M} = \M/\m_E\M$
has the desired form is obvious. \end{proof}

Define $\rho_j = \Qp \otimes_{\Zp} T_{\st,2}^{K}(\M_j)$, the potentially Barsotti-Tate
Galois representation associated to $\M_j$.

\begin{prop}
  The representation $\rho_j$ has inertial type $\om^n
  \oplus 1$, and $\det(\rho_j)$ is the product of the cyclotomic character,
  a finite order character of order prime to $p$, and an unramified character.
\end{prop}

\begin{proof}
Let $D$ denote the filtered module with descent data associated to
 $\rho_j$.
  We recall from the proof of \cite[Lem. 3.13]{sav04} that $D$ is equal
  to
the kernel of $N$ on $ \M_j[1/p]$.

Write $M = (W(k_2)[1/p] \otimes_{\Qp} E) \otimes_{S[1/p]} \M_j[1/p]$, the tensor product taken with respect
to the map $S[1/p] \rightarrow W(k_2)[1/p] \otimes_{\Qp} E$ sending
$u$ to $0$, and equip $M$ with the maps $N$ and
$\phi$ induced from $\M_j$ (so in particular $N=0$ on $M$), as
well as induced descent data $\widehat{g}$.  By
\cite[Prop. 6.2.1.1]{BreuilGriffiths} the canonical map $\M_j[1/p] \rightarrow M$ has a
unique $ W(k_2)[1/p]$-linear section $s : M \rightarrow \M_j[1/p]$ preserving $\phi$ and $N$;
then the same uniqueness argument as in the last paragraph of the proof of
Theorem \ref{thm:strdiv} shows that $s$ is an $E$-linear map and that
$s$ preserves descent data.

Recalling that $N=0$ on $M$,  we see that $D
= \im(s)$, so in particular $D$ has a $ W(k_2)[1/p]\otimes_{\Qp} E$-basis $v_1, v_2$ with $v_i =
s(g_i)$.  Since $s$ preserves descent data we have $\widehat{g} \cdot v_1 =
v_1 $ and $\widehat{g} \cdot v_2 = \om^n(g) v_2$.  The first part of
the proposition follows.

For the second part, if we regard $K_2 \otimes_{\Qp} E$ as an
$S[1/p]$-algebra via the map $u \mapsto \varpi$ then by
\cite[Prop. 6.2.2.1]{BreuilGriffiths} there is an isomorphism
$$f_{\varpi}: (K_2 \otimes_{\Qp} E) \otimes_{S[1/p]} \M_j[1/p] \cong D_{K_2} := {K_2}
\otimes_{W(k_2)[1/p]} D$$
that identifies the filtrations on both sides.  It follows that
$D_{K_2}$ is the free $(K_2 \otimes_{\Qp} E)$-module generated by
$f_{\varpi}(g_1)$ and $f_{\varpi}(g_2)$, and that $\Fil^1 D_{K_2}$ is
generated by $f_{\varpi}(h_1) = \varpi^J f_{\varpi}(g_1) + x_1
f_{\varpi}(g_2)$.  (One checks that $f_{\varpi}(h_1)$ and $f_{\varpi}(h_2)$
are scalar multiples of one another in $D_{K_2}$.)  In particular
$\Fil^1 D_{K_2}$ is a free submodule of rank one in $D_{K_2}$, and so all pairs of labeled
Hodge-Tate
weights of $\rho_j$ are $\{0,1\}$.  The claim follows from the note after Defintion~\ref{defn:barsotti tate lifts}, together
with the fact that $K_2/K$ is tamely ramified (so that the finite
order character has order prime to~$p$).
\end{proof}

\subsection{The niveau 2 case: conclusion of the
  proof} \label{sec:niveau_two_conclusion}
   We will now compute
$\rhobar_j = T_{\st,2}^{K}(\M_j/\m_E
\M_j)$, which by \cite[Prop
4.13]{sav04} is the reduction mod $p$ of
$\rho_j$.  More precisely we will compute $\rhobar_j |_{G_L}$ where
$L$ is the unramified quadratic extension of $K$ contained in
$K_2$.

As in Lemma~\ref{lem:connected_breuil_module}
let $\overline{\chi} : \Gal(K_2/L) \rightarrow k_E^{\times}$ be a
  character and let $\M(\overline{\chi})$ denote the rank one Breuil module
  with $k_E$-coefficients and descent data from $K_2$ to $L$ with
  generator $v$ and
$$ \Fil^1 \M(\overline{\chi}) =  {\M}(\overline{\chi}), \qquad
\phi_1(v) = v, \qquad \widehat{g}(v) = (1 \otimes
\overline{\chi}(g))v$$
for $g \in \Gal(K_2/L)$.  By Lemma~\ref{lem:connected_breuil_module}
we have
\begin{equation}
  \label{eq:tst_for_chars}
T_{\st,2}^L({\M}(\overline{\chi})) =
\overline{\chi}.
\end{equation}

Let $\overline{\M}^2_j$ denote the Breuil module $\overline{\M}_j =
\M_j/\m_E \M_j$ with its descent data restricted to $\Gal(K_2/L)$, so
that $T_{\st,2}^{L}(\overline{\M}^2_j) =
T_{\st,2}^{K}(\overline{\M}_j) |_{G_L}$.  We abuse notation and
also let $\omega_{\sigma}$ denote $\sigma \circ \omega_2 |_{G_L}$.

\begin{prop} \label{prop:reduction_of_mjbar} We have $\rhobar_j
  |_{G_L} = T_{\st,2}^{L}(\overline{\M}_j^2) \cong
  \omega_{\sigma_1}^{j+e} \oplus \omega_{\sigma_2}^{j+e}$.
\end{prop}

\begin{proof}
Let $e_{\sigma_1}, e_{\sigma_2} \in k_2 \otimes k_E$ denote the
idempotents corresponding to the embeddings $\sigma_1,\sigma_2 : k_2
\hookrightarrow k_E$, so that $e_{\sigma_i}(a \otimes 1) =
e_{\sigma_i}(1 \otimes \sigma_i(a))$.  Suppose $\{\alpha, \beta\} =
\{1,2\}$.
If one ignores descent data, one checks that there is a map
$f_{\beta} : \M(\overline{\chi}) \rightarrow \overline{\M}_j^2 $
obtained by sending
$$ v \mapsto u^{p(j+e)} e_{\sigma_{\alpha}}\overline{g}_1 + u^{p(pe-j)}
e_{\sigma_{\beta}} \overline{g}_2 .$$
 In order that this map be compatible with descent data, one checks that it is
necessary and sufficient that $e_{\sigma_{\alpha}}(1 \otimes
\overline{\chi}) = e_{\sigma_{\alpha}}(\omega_2^{p(j+e)}\otimes 1)$,
i.e., $\overline{\chi} = \omega_{\sigma_{\alpha}}^{p(j+e)} = \omega_{\sigma_{\beta}}^{j+e}$.
We therefore have a map
$$ f_1 \oplus f_2 : \M(\omega_{\sigma_1}^{j+e}) \oplus \M(\omega_{\sigma_2}^{j+e})
\rightarrow \overline{\M}_j^2.$$
Note that $\ker(f_1 \oplus f_2)$ does not contain any free
$k_2[u]/u^{e_2p}$-submodules (this amounts to the fact that $p(j+e)$
and $p(pe-j)$ are both smaller than $pe_2$);
by \cite[Prop
8.3]{SavittCompositio}
we deduce that $f_1 \oplus f_2$ induces an isomorphism on
generic fibres, and the proposition follows from \eqref{eq:tst_for_chars}.
\end{proof}

Finally we have the following.

\begin{proof}[Proof of Lemma~\ref{lem:existence of lifts in the tame niveau 2 case}]
Take $K = F_{\p}$ in the discussion of this and the previous subsection. Twisting by a suitable power of the (Teichm\"{u}ller lift of) a
 fundamental  character of level one, we may
 assume $m=0$.   Setting $j = n + (p-1)(e-x)$, one checks that
$\omega_{\sigma_1}^{j+e} = \omega_{\sigma_1}^{n+x}
\omega_{\sigma_2}^{e-x}$
and similarly for $\omega_{\sigma_2}^{j+e}$.  It follows from
Proposition \ref{prop:reduction_of_mjbar} that $\rhobar$ is an
unramified twist of $\rhobar_j$.   Note that a twist of $\rho_j$ by a
suitable unramified character will be parallel; twisting this by a
suitable finite-order
unramified character of order prime to $p$  gives a parallel
potentially Barsotti-Tate lift
of $\rhobar$ that is an unramified twist of~$\rho_j$.  Since $\rho_j$ has
type $\om^{n} \oplus 1$ we are done.
\end{proof}

\begin{rem}
  The reader may find it unnatural that although $\rho_j$ becomes
Barsotti-Tate over $K_1 = K(\pi^{1/(p-1)})$, we instead work with a
strongly divisible module over $K_2$ for $\rho_j$ (because our method
for computing $\rhobar_j |_{G_L}$ requires it).  One can certainly
write down the strongly divisible module over $K_1$ instead (just
replace $J$ and $e_2$ with $j$ and $e_1$ throughout the construction
of $\M_j$), whose reduction mod $p$ corresponds to a group scheme $\G$
over $\OO_{K_1}$ with generic fibre descent data from $K_1$ to $K$.
One can then hope to show directly, by extending the methods of
\cite[Sec. 5.4]{bcdt}, that $\G \times_{\OO_{K_1}} \OO_{K_2}$
 (with generic fibre descent data from $K_2$ to $K$) corresponds to
our $\overline{\M}_j$.  However, this last step would require at least
several extra pages of rather technical work, so we prefer to proceed
as above instead.
\end{rem}

\section{The main theorems}\label{main results}Recall that we are
assuming that $F$ is a totally real field in which the prime $p$ is
totally ramified. We now prove the main
results of this paper, by combining the techniques of earlier sections
with the lifting machinery of Khare-Wintenberger, as interpreted by
Kisin. In particular, we use the following result.

\begin{thm}
  \label{thm:existence-of-global-lifts}
  Suppose that $p>2$ and that $\rhobar:G_F\to\GL_2(\Fpbar)$ is
  modular. Assume that $\rhobar|_{G_{F(\zeta_p)}}$ is irreducible. If $p=5$ and the projective image of $\rhobar$ is
  isomorphic to $\PGL_2(\F_5)$, assume further that
  $[F(\zeta_p):F]=4$.
  \begin{itemize}
  \item

Suppose that $\rhobar|_{G_{F_\p}}$ has a
  non-ordinary parallel potentially Barsotti-Tate lift of type
  $\tilde{\omega}^{m+n}\oplus\tilde{\omega}^{m}$. Then $\rhobar$ has a
  modular lift which is parallel  potentially Barsotti-Tate of type
  $\tilde{\omega}^{m+n}\oplus\tilde{\omega}^{m}$ and non-ordinary.
\item Suppose that $\rhobar$ has an ordinary modular lift. Suppose also that $\rhobar|_{G_{F_\p}}$ has an
  ordinary parallel potentially Barsotti-Tate lift of type
  $\tilde{\omega}^{m+n}\oplus\tilde{\omega}^{m}$. Then $\rhobar$ has a
  modular lift which is parallel potentially Barsotti-Tate of type
  $\tilde{\omega}^{m+n}\oplus\tilde{\omega}^{m}$ and ordinary.
  \end{itemize}
\
\end{thm}
\begin{proof}
  This is a special case of Corollary 3.1.7 of \cite{gee061}.
\end{proof}
In combination with the local computations of section \ref{sec:
  crystalline chars}, this shows us that in most cases if $\rhobar|_{G_{F_\p}}$ is
semisimple, and $\sigma_{m,n}\in
W^?(\rhobar)$, then $\rhobar$ has a modular lift which is parallel potentially Barsotti-Tate of
type $\tilde{\omega}^{m+n}\oplus\tilde{\omega}^n$ (the exceptional
cases being those where the only local lifts of this type are
ordinary, and $\rhobar$ is not known to have an ordinary lift). By Lemma
\ref{lem:typesversusweights}, this means that $\rhobar$ is modular of
weight $\sigma_{m,n}$ or $\sigma_{m+n,p-1-n}$. However, we can
frequently guarantee that this lift is non-ordinary, and the weight
cycling techniques of section \ref{sec: weight cycling} then give the
following far more useful result.
\begin{thm}\label{thm:non-ordinary lifts and non-ordinary weights}Suppose that $p>2$ and that $\rhobar:G_F\to\GL_2(\Fpbar)$ is
  modular. Assume that $\rhobar|_{G_{F(\zeta_p)}}$ is irreducible. If $p=5$ and the projective image of $\rhobar$ is
  isomorphic to $\PGL_2(\F_5)$, assume further that
  $[F(\zeta_p):F]=4$. If
  $\rhobar|_{G_{F_{\p}}}$ has a non-ordinary parallel potentially Barsotti-Tate
  lift of type $\tilde{\omega}^{m+n}\oplus\tilde{\omega}^{m}$, $0\leq
  n\leq p-1$, then
  $\rhobar$ is modular both of weight $\sigma_{m,n}$ and of weight $\sigma_{m+n,p-1-n}$.
\end{thm}
\begin{proof} By Theorem \ref{thm:existence-of-global-lifts}, there is a modular
  lift of $\rhobar$ which is parallel potentially Barsotti-Tate of type
  $\tilde{\omega}^{m+n}\oplus\tilde{\omega}^{m}$, and which is
  non-ordinary. The result follows from Proposition \ref{prop: non-ordinary global  lift gives both weights via cycling}.
\end{proof}

We now extract some consequences from this result. Suppose that
$\rhobar|_{G_{F_\p}}$ is semisimple. Then we
have already proved that $W(\rhobar)$, the set of weights $\sigma$ for
which $\rhobar$ is modular of weight $\sigma$, is contained in
$W^?(\rhobar)$ (this is Corollary \ref{cor:tame case, necessary conditions on galois rep}). We can frequently deduce the converse implication,
showing that if $\sigma_{m,n}\in W^?(\rhobar)$ then $\sigma_{m,n}\in
W(\rhobar)$.  In particular we have the following.

\begin{corollary}\label{cor:high ramification or irreducible}Suppose that $p>2$ and that $\rhobar:G_F\to\GL_2(\Fpbar)$ is
  modular. Assume that $\rhobar|_{G_{F(\zeta_p)}}$ is irreducible. If $p=5$ and the projective image of $\rhobar$ is
  isomorphic to $\PGL_2(\F_5)$, assume further that
  $[F(\zeta_p):F]=4$. Suppose that $\rhobar|_{G_{F_{\p}}}$ is
  semisimple.
  \begin{enumerate}
  \item  If $\rhobar|_{G_{F_{\p}}}$ is irreducible or $e \ge p$, then
    $\rhobar$ is modular of weight $\sigma$ if and only if $\sigma\in
    W^{?}(\rhobar)$.
 \item If $e \le p-1$, then
   $\rhobar$
is modular of weight $\sigma$ if and only if $\sigma \in
W^{?}(\rhobar)$
except possibly if $\sigma = \sigma_{m,n}$
and
 $$\rhobar|_{I_{F_{\p}}}\cong \left(\begin{matrix}\omega^{m+n+e} & 0\\0& \omega^{m}
\end{matrix}\right)$$ with $n+e \leq p-1$ or $n=p-1$.
  \end{enumerate}

\end{corollary}
\begin{proof} Write $\sigma = \sigma_{m,n}$.  As already remarked, the ``only if'' direction is
  Corollary \ref{cor:tame case, necessary conditions on galois rep}.
For the ``if'' direction, 
by Theorem \ref{thm:non-ordinary lifts and non-ordinary weights} it
suffices to be able to produce a non-ordinary parallel potentially
Barsotti-Tate lift of $\rhobar |_{G_{F_\p}}$ of type
$\tilde{\omega}^{m+n}\oplus\tilde{\omega}^{m}$.   Since we assume $\sigma_{m,n} \in
W^?(\rhobar)$, the representation $\rhobar |_{I_{F_\p}}$ is as in Definition~\ref{def: the predicted weights}.
Then existence of the desired lift follows from 
 Lemma
 \ref{lem:existence of lifts in the tame niveau 1 case} when $\rhobar
 |_{G_{F_\p}}$ is reducible, and from Lemma
  \ref{lem:existence of lifts in the tame niveau 2 case} when $\rhobar
  |_{G_{F_\p}}$ is irreducible.  Note that in part (2), the
  exceptional cases are precisely the ones where none of the lifts
  provided by Lemma~\ref{lem:existence of lifts
    in the tame niveau 1 case} are non-ordinary.
\end{proof}

Note that there are at most four exceptional cases in part (2) of
Corollary~\ref{cor:high ramification or irreducible}: there are
two ways of ordering the diagonal characters, and each ordering will
correspond either to one or two values of $n$ (if $n \not\equiv 0$ or $n\equiv 0 \pmod{p-1}$ respectively).

In fact, if we assume in addition that $\rhobar$ has an ordinary
modular lift, then we are able to dispose of most of these exceptional
cases. This relies on something of a combinatorial coincidence; it
turns out that in most cases where $\sigma_{m,n}\in W^?(\rhobar)$ but
$\rhobar|_{G_{F_\p}}$ has only ordinary lifts of type
$\tilde{\omega}^{m+n}\oplus\tilde{\omega}^{m}$, then
$\sigma_{m+n,p-1-n}\notin W^?(\rhobar)$, so the combination of Theorem
\ref{thm:existence-of-global-lifts}, Lemma
\ref{lem:typesversusweights} and Corollary \ref{cor:tame case,
  necessary conditions on galois rep} shows that in fact
$\sigma_{m,n}\in W(\rhobar)$.
\begin{corollary}
  \label{cor:lowramificationniveauonemainresult}Suppose that $p>2$ and that $\rhobar:G_F\to\GL_2(\Fpbar)$ is
  modular. Assume that $\rhobar|_{G_{F(\zeta_p)}}$ is irreducible. Suppose that $e\leq p-1$ and
  that $\rhobar|_{G_{F_\p}}$ is semisimple and
  reducible. Suppose further that $\rhobar$ has an ordinary modular lift. If $p=5$ and the projective image of $\rhobar$ is
  isomorphic to $\PGL_2(\F_5)$, assume further that
  $[F(\zeta_p):F]=4$. Suppose that $\sigma_{m,n}\in W^?(\rhobar)$. If $n=p-1$,
  suppose that $$\rhobar|_{I_{F_{\p}}}\ncong
  \left(\begin{matrix}\omega^{m+e} & 0\\0&
      \omega^{m} \end{matrix}\right).$$ Then $\rhobar$ is modular
of weight $\sigma_{m,n}$.

\end{corollary}
\begin{proof}
If $\rhobar|_{G_{F_\p}}$ has a non-ordinary parallel
  potentially Barsotti-Tate lift of type
  $\tilde{\omega}^{m+n}\oplus\tilde{\omega}^{m}$, then
by Corollary
\ref{cor:high ramification or irreducible}, $\rhobar$ is
  modular of weight $\sigma_{m,n}$.

Suppose now that $\rhobar|_{G_{F_\p}}$ does not have a non-ordinary
parallel  potentially Barsotti-Tate lift of type
  $\tilde{\omega}^{m+n}\oplus\tilde{\omega}^{m}$. By Lemma \ref{lem:existence of lifts in the tame niveau 1
    case}, we must have
 $$\rhobar|_{I_{F_{\p}}}\cong \left(\begin{matrix}\omega^{m+n+e} & 0\\0& \omega^{m}
\end{matrix}\right).$$
 Furthermore, either $n+e\leq p-1$, or
$n=p-1$. The second case is precisely the case excluded by the
statement of this corollary.

Thus we must have $n+e\leq p-1$. By Lemma \ref{lem:existence of lifts in
  the tame niveau 1 case}, $\rhobar|_{G_{F_\p}}$ has a parallel potentially
Barsotti-Tate lift of type
$\tilde{\omega}^{m+n}\oplus\tilde{\omega}^{m}$, so that by the
assumption that $\rhobar$ has an ordinary modular lift, Theorem
\ref{thm:existence-of-global-lifts}, and Lemma
\ref{lem:typesversusweights}, $\rhobar$ is modular of weight
$\sigma_{m,n}$ or $\sigma_{m+n,p-1-n}$. If $n=0$ then we may conclude
further that $\rhobar$ is modular of weight $\sigma_{m,0}$. Assume for
the sake of contradiction that $\rhobar$ is not modular of weight
$\sigma_{m,n}$, so that we
may assume that $n\neq 0$, $n+e\leq p-1$, and $\sigma_{m+n,p-1-n}\in
W(\rhobar)$. In particular, by Corollary \ref{cor:tame case, necessary
  conditions on galois rep} we have $\sigma_{m+n,p-1-n}\in
W^?(\rhobar)$, and we also know that
$e<p-1$ (because $e\leq p-1-n<p-1$). Now,
examining the definition of $W^?(\rhobar)$ (Definition \ref{def: the
  predicted weights}), we see that we must have  $$\rhobar|_{I_{F_{\p}}}\cong \left(\begin{matrix}\omega^{m+x} & 0\\0& \omega^{m+n+e-x}

\end{matrix}\right)$$for some $1\leq x\leq e$. Comparing with the
expression above, we see that either $\omega^{m+x}=\omega^m$ or
$\omega^{m+x}=\omega^{m+n+e}$. The first possibility requires $x\equiv
0\text{ (mod }p-1)$, a contradiction as $1\leq x\leq e<p-1$. The
second requires  $x\equiv
n+e\text{ (mod }p-1)$, which is a contradiction because $1\leq x\leq
e<n+e\leq p-1$. The result follows.
\end{proof}





\bibliographystyle{amsalpha} 
\bibliography{tobybib08} 
\end{document}